\documentclass[12pt]{amsart}
\usepackage[dvips]{color}
\usepackage{amsmath}
\usepackage{amsxtra}
\usepackage{amscd}
\usepackage{amsthm}
\usepackage{amsfonts}
\usepackage{amssymb}
\usepackage{eucal}
\usepackage{epsfig}
\usepackage{graphics}

\numberwithin{equation}{section}
\newtheorem{thm}{Theorem}[section]
\newtheorem{prop}[thm]{Proposition}
\newtheorem{lem}[thm]{Lemma}

\newtheorem{cor}[thm]{Corollary}
\newtheorem{conj}[thm]{Conjecture}

\newcommand{\nn}{\nonumber}

\newcommand{\bra}[1]{\langle #1 |}        
\newcommand{\ket}[1]{{| #1 \rangle}}      
\newcommand{\br}[1]{{\langle #1 \rangle}}  

\newcommand{\C}{{\mathbb C}}
\newcommand{\Z}{{\mathbb Z}}

\newcommand{\A}{{\mathcal A}}
\newcommand{\B}{{\mathcal B}}
\newcommand{\E}{{\mathcal E}}
\newcommand{\F}{\mathcal F}

\newcommand{\cM}{\mathcal{M}}
\newcommand{\cN}{\mathcal{N}}
\newcommand{\cO}{\mathcal{O}}
\newcommand{\cP}{\mathcal{P}}
\newcommand{\cR}{\mathcal{R}}
\newcommand{\cT}{\mathcal{T}}
\newcommand{\cX}{\mathcal{X}}

\newcommand{\Bb}{\overline{\B}}
\newcommand{\Rb}{\overline{R}}
\newcommand{\chib}{{\chi}}

\newcommand{\gt}{\widetilde{\mathfrak{g}}}

\newcommand{\bs}{\boldsymbol}

\newcommand{\g}{\mathfrak{g}}
\newcommand{\gl}{\mathfrak{gl}}
\newcommand{\rg}{\mathfrak{r}}
\newcommand{\mb}{\mathbf{\mathfrak{m}}}

\newcommand{\slt}{\mathfrak{sl}_2}

\newcommand{\Tb}{\mathfrak{T}}

\newcommand{\af}{\textsf{A}}

\newcommand{\bla}{{\boldsymbol \la}}

\newcommand{\la}{\lambda}

\newcommand{\Ga}{\Gamma}
\newcommand{\ga}{\gamma}
\newcommand{\al}{\alpha}

\newcommand{\bc}{{\bf{c}}}
\newcommand{\bh}{{\bf{h}}}
\newcommand{\bm}{{\bf{m}}}

\newcommand{\bp}{{\bf{p}}}

\newcommand{\End}{\mathop{\rm End}}
\newcommand{\Hom}{\mathop{\rm Hom}}
\newcommand{\id}{{\rm id}}

\newcommand{\pdeg}{\mathop{\rm pdeg}}
\newcommand{\hdeg}{\mathop{\rm hdeg}}
\newcommand{\Ob}{\mathrm{Ob}\,}
\newcommand{\Rep}{\mathrm{Rep}}

\newcommand{\Tr}{{\rm Tr}}

\newcommand{\wt}{{\rm wt}}
\newcommand{\elwt}{\ell{\rm wt}}
\newcommand{\gge}{\geqslant}
\newcommand{\lle}{\leqslant}

\newcommand{\Cp}{C^{\perp}}

\newcommand{\ep}{e^{\perp}}
\newcommand{\fp}{f^{\perp}}
\newcommand{\hp}{h^{\perp}}
\newcommand{\ad}{\mathop{\mathrm{ad}}}

\newcommand{\Fin}{\cO^{fin}_{\B^\perp}}

\begin{document}
\begin{title}[finite type modules for quantum toroidal $\mathfrak{gl}_1$]
{finite type modules and Bethe Ansatz \\
for quantum toroidal $\mathfrak{gl}_1$}
\end{title}
\author{B. Feigin, M. Jimbo, T. Miwa and E. Mukhin}
\address{BF: National Research University Higher School of Economics, 
Russian Federation, 
Russia, Moscow,  101000,  Myasnitskaya ul., 20 
and Landau Institute for Theoretical Physics,
Russia, Chernogolovka, 142432, pr.Akademika Semenova, 1a
}
\email{bfeigin@gmail.com}
\address{MJ: Department of Mathematics,
Rikkyo University, Toshima-ku, Tokyo 171-8501, Japan}
\email{jimbomm@rikkyo.ac.jp}
\address{TM: Institute for Liberal Arts and Sciences,
Kyoto University, Kyoto 606-8316,
Japan}\email{tmiwa@kje.biglobe.ne.jp}
\address{EM: Department of Mathematics,
Indiana University-Purdue University-Indianapolis,
402 N.Blackford St., LD 270,
Indianapolis, IN 46202, USA}\email{emukhin@iupui.edu}

\begin{abstract} 
We study highest weight representations of the Borel subalgebra of 
the quantum toroidal $\gl_1$ algebra with finite-dimensional weight spaces. 
In particular, we develop the $q$-character theory for such modules.
We introduce and study the subcategory of `finite type' modules.
By definition, a module over the Borel subalgebra is finite type if
the Cartan like current $\psi^+(z)$ has a finite number of eigenvalues,
even though the module itself can be infinite dimensional.

We use our results to diagonalize the transfer matrix $T_{V,W}(u;p)$ 
analogous to those of the six vertex model. In our setting $T_{V,W}(u;p)$ acts in a tensor product $W$
of Fock spaces and $V$ is a highest weight module over the Borel subalgebra of quantum toroidal $\gl_1$ 
with finite-dimensional weight spaces. Namely we show that for a special choice of finite type modules $V$ the corresponding transfer 
matrices, $Q(u;p)$ and $\cT(u;p)$, 
are polynomials in $u$
and satisfy a two-term $TQ$ relation. We use this relation to prove 
the Bethe Ansatz equation for the 
zeroes of the eigenvalues of $Q(u;p)$.  Then 
we show that the eigenvalues of $T_{V,W}(u;p)$ are given by an appropriate substitution of eigenvalues of 
$Q(u;p)$ into the $q$-character of $V$.
\end{abstract}

\maketitle 

\section{Introduction}

The six vertex model 
is a well known representative example of quantum integrable systems. 
Its integrability is attributed to a large symmetry 
under the quantum loop algebra $U_q\widetilde\slt$. 
This setting easily generalizes to an arbitrary simple Lie algebra $\g$;
with each choice of $U_q\widetilde\g$ and its representation, 
an analog of the six vertex model is defined. 
Algebra $U_q\gt$ is a quantization of the algebra of 
loops $\g[x^{\pm1}]$ with values in $\g$. 
We may then na{\"i}vely ask the following question: 
what if we replace the loop algebra in one variable $\g[x^{\pm1}]$ 
with that of two variables $\g[x^{\pm1},y^{\pm1}]$?

Quantum version of loop algebras in two variables are 
the quantum toroidal algebras, introduced 
by Ginzburg, Kapranov and Vasserot \cite{GKV}. 
The last decade has seen novel developments 
in connection with geometric representation theory and gauge theory,
and there are now revived interests in this subject, see e.g.
\cite{FHHSY}, \cite{AFS}, \cite{MO}, \cite{NPS}.

In this article we are concerned with the 
quantum toroidal algebra of type $\gl_1$.
We denote it by $\E$. 
Algebra $\E$ has a universal $R$ matrix, which allows us to consider
a family of commuting transfer matrices  $T_{V,W}(u;p)$
analogous to those 
of the six vertex model $T_{6v}(u;p)$. 
Here $u$ is a `spectral parameter', $p$ is a `twist parameter', 
$V$ denotes the `auxiliary space', and  $W$ denotes the `quantum space' 
on which the transfer matrices act.  

In the original six vertex model, transfer matrices are operators derived from an auxiliary space 
$V=\C^2$,
which act on the quantum space $W=(\C^2)^{\otimes N}$.
All $\C^2$ are considered as 
two dimensional representations of $U_q(\widetilde\slt)$. In the case of 
the quantum toroidal algebra, 
the quantum space $W$ is an $N$ fold tensor product of Fock spaces $\F(v)^{\otimes N}$. 
The Fock module $\F(v)$ is  an infinite dimensional space which has a basis labeled by all partitions. 
The auxiliary space $V$ is any highest weight representation of the Borel subalgebra of quantum toroidal $\gl_1$ 
with finite-dimensional weight spaces.

Besides the intrinsic interest on its own, 
this problem has a close connection to a topic in 
conformal field theory (CFT).  
In a seminal series of papers \cite{BLZ1}--\cite{BLZ3}
Bazhanov, Lukyanov and Zamolodchikov introduced and studied 
a commutative subalgebra inside the enveloping algebra of the Virasoro algebra. 
This subalgebra gives commuting operators called integrals of motion (IM), 
which act on Virasoro Verma modules. 
Dorey and Tateo \cite{DT} and Bazhanov et al. \cite{BLZ4} 
discovered a remarkable connection between 
the spectra of the $Q$ operators, which are certain generating functions of IM, 
and spectral determinants of a certain family of one-dimensional Schr{\"o}dinger operators. 
Subsequently a $q$ analog of IM was considered by Kojima, Shiraishi, 
Watanabe and one of the present authors \cite{FKSW}. 
They introduced a family of operators $I_n$, $n=1,2,\cdots$, in terms of 
contour integrals involving currents of the 
deformed Virasoro algebra, and showed their commutativity by 
direct computation. 
As it turns out, their first Hamiltonian $I_1$ is obtained 
from the transfer matrix $T_{\F,W}(u;p)$, where $\F=\F(1)$,
for the algebra $\E$ 
by taking the first term in the expansion as $u\to 0$. 
Thus it is natural to expect that the six vertex type model 
associated with $\E$ gives the same integrable system defined by 
the $q$-deformed  IM of \cite{FKSW}. 
In the limit to CFT, 
Litvinov \cite{L} put forward a conjectural Bethe Ansatz equation 
which describes the spectrum of this system.  

In our previous work \cite{FJMM2}, we showed that the eigenvalues 
of $I_1$ are indeed given in terms of the sum of the Bethe roots.
In this paper we use a method different from the one in 
\cite{FJMM2}, and obtain the spectrum of the transfer matrix $T_{\F,W}(u;p)$,
and more generally $T_{V,W}(u;p)$, in full. 
Namely we adopt Baxter's method of $TQ$ relation \cite{Ba}, 
which we recall here.  

In the case of the six vertex model, the transfer matrix 
$T_{6v}(u;p)$ with a suitable normalization 
is a polynomial in the spectral parameter $u$.  
Baxter showed that there is another matrix
$Q_{6v}(u;p)$, which depends polynomially in $u$,  
commutes with $T_{6v}(u;p)$, and satisfies the relation
\begin{align*}
T_{6v}(u;p)Q_{6v}(u;p)=a(u)Q_{6v}(q^{-2} u;p)+p\, d(u)Q_{6v}(q^{2} u;p)\,, 
\end{align*}
where $a(u)$ and $d(u)$ are known scalar polynomials. 
This relation can be viewed as the one obeyed by 
eigenvalues of $T_{6v}(u;p)$ and $Q_{6v}(u;p)$. 
Denote the eigenvalues by the same letters. Then  
the $TQ$ relation implies the Bethe equation 
$0=a(\zeta_i)Q_{6v}(q^{-2} \zeta_i;p)+p\, d(\zeta_i)Q_{6v}(q^{2}\zeta_i;p)$ for  
the roots $\{\zeta_i\}$ of $Q_{6v}(u;p)$.
Once $Q_{6v}(u;p)$ is known, 
the corresponding eigenvalue $T_{6v}(u;p)$ is given in turn as 
\begin{align*}
T_{6v}(u;p)=a(u)\frac{Q_{6v}(q^{-2}u;p)}{Q_{6v}(u;p)}
\left(1+\mathfrak{a}_{6v}(u;p)\right)\,,
\quad \mathfrak{a}_{6v}(u;p)=p
\frac{d(u)}{a(u)}\frac{Q_{6v}(q^{2}u;p)}{Q_{6v}(q^{-2}u;p)}\,.
\end{align*}

Unlike the case of $U_q\widetilde\slt$, 
the $TQ$ relation in the toroidal case is written in terms of $Q(u;p)$ and 
a new, auxiliary matrix $\cT(u;p)$ different from $T_{\F,W}(u;p)$.
The relation takes the form
\begin{align}\label{MMM}
\cT(u;p) Q(u;p) =a(u)\prod_{s=1}^3Q(q^{-1}_s u;p)
+p\,d(u)\prod_{s=1}^3Q(q_su;p)\,,
\end{align}
where $q_s$ ($s=1,2,3$) are the parameters of the algebra $\E$
satisfying $q_1q_2q_3=1$.  
The original transfer matrix $T_{\F,W}(u;p)$ 
is then given by an infinite series
\begin{align}
T_{\F,W}(u;p)=\frac{Q(q_2^{-1}u;p)}{Q(u;p)}
\sum_{\lambda}
\prod_{\square\in\lambda}\mathfrak{a}(q^{-\square} u;p)\,,
\quad
\mathfrak{a}(u;p)
=p\frac{d(u)}{a(u)}\prod_{s=1}^3\frac{Q(q_su;p)}{Q(q^{-1}_s u;p)} \,.
\label{T-eigv}
\end{align}
Here
$\lambda$ 
runs over all partitions, 
$\square=(i,j)$ runs over the nodes
of $\lambda$ and $q^{\square}=q_3^{i-1}q_1^{j-1}$.
See Theorem \ref{thm:eigv} below. 
Construction of $Q(u;p)$ and $\cT(u;p)$, and 
formula \eqref{T-eigv} for the transfer matrix eigenvalues, 
are the main results of this paper. 

Our approach is based on representation theory. 
Bazhanov et al. \cite{BLZ3} showed that  
the matrix $Q_{6v}(u;p)$ is a transfer matrix 
of an appropriate representation of 
the Borel subalgebra of $U_q\widetilde\slt$.  
Frenkel and Hernandez \cite{FH} generalized this construction 
to an arbitrary quantum loop algebra $U_q\widetilde\g$ of 
non-twisted type, and obtained an expression of the transfer matrix
in terms of appropriate analogs of $Q_{6v}(u;p)$. 
In the present article, we construct the operator $Q(u;p)$
in the setting of the quantum toroidal algebra $\E$. 
We are able to add the next step: we also
construct an operator $\cT(u;p)$ which 
satisfies the two-term $TQ$ relation \eqref{MMM},
and therefore prove the Bethe equation for the 
zeroes of the eigenvalues of $Q(u;p)$, 
see Theorem \ref{thm:eigv}. 
  
Our construction is based on the study of infinite dimensional modules of the Borel subalgebra
on which the Cartan like generator $\psi^+(z)$ has
a finite number $k$ of distinct eigenvalues. 
We say that such a module is $k$-finite.
Operators $Q(u;p)$ and $\cT(u;p)$ are transfer matrices constructed from 
$1$-finite and $2$-finite modules, respectively. The two-term $TQ$ relation \eqref{NM-MMM} is a short exact sequence in the Grothendieck ring 
of representations of the Borel subalgebra.

We provide a grading 
similar to the one constructed in \cite{FH}, 
see Propositions \ref{prop:gradingM} and \ref{prop:gradingN},
 and use it for proving that $Q(u;p)$ and $\cT(u;p)$ are polynomials, see Propositions \ref{Q pol},  \ref{T pol}. 
Our construction of the grading and 
the proof of polynomiality are different from those in \cite{FH}.

In order to express the eigenvalues of arbitrary transfer matrices  
$T_{V,W}(u;p)$, we develop the theory of $q$-characters 
for representations of the Borel subalgebra. 
We prove some properties of the $q$-characters, 
see Section \ref{q char 2}, and use it to study finite type modules. 
We give a classification of $1$-finite modules, see Proposition \ref{1}. 
We also give conjectures including the one on the cluster algebra structure 
of the Grothendieck ring of the category of modules of finite type, see Section \ref{conj}. 
Then the matrix $T_{V,W}(u;p)$ (or the corresponding eigenvalues) is described by an appropriate 
substitution of $Q(u;p)$ (or eigenvalues of $Q(u;p)$) into the $q$-character of $V$, 
see Proposition \ref{prop:TQ} and Corollary \ref{V Q w}.

While we write our statements for the case of the quantum space $W$ being a tensor product of Fock spaces, one can easily repeat, with obvious changes, for the case when $W$ is an arbitrary tensor product of irreducible highest weight modules of the (dual) Borel subalgebra. 

\medskip

Similar two term relations hold in the case of quantum affine algebras
as well. We plan to discuss this issue
in a separate publication. 

\medskip

The text is organized as follows. 
In Section \ref{sec:toroidal-alg}, we collect basic definitions
and facts concerning the quantum toroidal $\gl_1$ algebra $\E$ 
and its Borel subalgebras.  
In Section \ref{sec:modules} we discuss modules over $\E$ and its 
Borel subalgebra $\B^\perp$, and the theory of $q$-characters.
Section \ref{sec:finite-type} is devoted to 
the study of finite type modules over $\B^\perp$.   
In the last Section \ref{sec:Bethe} we introduce the transfer matrices, establish the Bethe ansatz equation and write the spectra of transfer matrices. 
Several technical points are collected in the Appendix.

\section{Quantum toroidal $\gl_1$}\label{sec:toroidal-alg}

In this section we summarize basic definitions and facts 
concerning the quantum toroidal algebra of type $\gl_1$.

\subsection{Algebra $\E$}
Throughout the text we fix 
complex numbers $q, q_1,q_2,q_3$ satisfying $q_2=q^2$ and $q_1q_2q_3=1$. 
We assume further that, for integers $l,m,n\in\Z$,  
$q_1^lq_2^mq_3^n=1$ holds only if $l=m=n$. 

The quantum toroidal algebra of type $\gl_1$, which we denote by $\E$,  
is a bi-graded $\C$ algebra generated by 
$e_n,\ f_n$ ($n\in \Z$), $h_r$ ($r\in \Z\backslash\{0\}$)
and invertible elements $C,\ \Cp,\ D,\ D^\perp$, with bi-degrees
\begin{align*}
&\deg e_n=(1,n),\quad  \deg f_n=(-1,n),\quad  \deg h_r=(0,r)\,,
\\
&\deg x=(0,0)\quad (x=C,\ \Cp,\ D,\ D^\perp).
\end{align*}
For a homogeneous element $x\in\E$  
with $\deg x=(\nu_1,\nu_2)$, 
we say that $x$ has {\it principal degree} $\nu_1$ and {\it homogeneous degree} 
$\nu_2$, and write 
$\mathrm{pdeg}\, x=\nu_1$, $\mathrm{hdeg}\, x=\nu_2$. 
 
Elements $C,C^{\perp}$ are central in $\E$. Elements
$D,D^\perp$ count the degrees of an element $x\in\E$: 
$DxD^{-1}=q^{-\mathrm{hdeg}\, x}x$, 
$D^\perp x(D^\perp)^{-1}=q^{\mathrm{pdeg}\, x}x$.  
The rest of the defining relations are given as follows. 
Using the symbols
\begin{align*}
& g(z,w)=(z-q_1w)(z-q_2w)(z-q_3w),\\
&\kappa_r=(1-q_1^r)(1-q_2^r)(1-q_3^r)\,,
\end{align*}
we have 
\begin{gather*}
[h_{r},h_{s}]=
\delta_{r+s,0}\,\frac{1}{r} \frac{C^{r}-C^{-r}}{\kappa_r}\,,
\\ 
[h_{r},e_n]=-\frac{1}{r}\, e_{n+r}\, C^{(-r-|r|)/2}\,,
\quad 
[h_{r},f_n]=
\frac{1}{r}\, f_{n+r}\, C^{(-r+|r|)/2}\,,
\\ 
[e(z),f(w)]=\frac{1}{\kappa_1}
(\delta\bigl(\frac{Cw}{z}\bigr)\psi^+(w)
-\delta\bigl(\frac{Cz}{w}\bigr)\psi^-(z)),\\
g(z,w)e(z)e(w)+g(w,z)e(w)e(z)=0, \\
g(w,z)f(z)f(w)+g(z,w)f(w)f(z)=0,\\
[e_n,[e_{n-1},e_{n+1}]]=0\,,\quad 
[f_n,[f_{n-1},f_{n+1}]]=0\,,
\end{gather*}
for all $n\in \Z$ and $r,s\in\Z\backslash\{0\}$. 
Here we use 
generating series $e(z) =\sum_{n\in \Z} e_{n}z^{-n}$, $f(z) =\sum_{n\in\Z} f_{n}z^{-n}$, and $\psi^{\pm}(z)=\sum_{\pm n\ge0} \psi^\pm_{n}z^{-n}$ are given by
\begin{align*}
&\psi^{\pm}(z) = (\Cp)^{\mp 1} 
\exp\bigl(\sum_{r=1}^\infty \kappa_r h_{\pm r}z^{\mp r}\bigr)\,.
\end{align*}
The relations between the $h_r$'s and $e(z),f(z)$ can also be written as 
the bilinear relations
\begin{align*}
&g(C^{(1\pm1)/2}z,w)\psi^\pm(z)e(w)+g(w,C^{(1\pm1)/2}z)e(w)\psi^\pm(z)=0
\,,\\ 
&g(w,C^{(-1\pm1)/2}z)\psi^\pm(z)f(w)+g(C^{(-1\pm1)/2}z,w)f(w)\psi^\pm(z)=0\,.
\end{align*}

It is easy to see that the elements $e_0,f_0,h_{\pm1}$ 
together with $C,C^\perp,D,D^\perp$
generate $\E$. 

Let $\E=\oplus_{\nu_1,\nu_2\in\Z}\E_{\nu_1,\nu_2}$ be the decomposition 
of algebra $\E$ into bigraded subspaces. 
For $p\in \Z$, we set 
\begin{align*}
&\displaystyle{\E_{\gge p}= 
\bigoplus_{\nu_1\ge p\atop\nu_2\in\Z}\E_{\nu_1,\nu_2}}, \quad
\displaystyle{\E_{\lle p}=
\bigoplus_{\nu_1\le p\atop\nu_2\in\Z}\E_{\nu_1,\nu_2}}. 
\end{align*}

\subsection{Elliptic Hall algebra}

Algebra $\E$ is a quantum version of 
the Lie algebra of currents $\gl_1[x^{\pm1},y^{\pm1}]$ 
with a two dimensional central extension.    
It was originally introduced in \cite{BS} 
and was called elliptic Hall algebra, see also \cite{M}, \cite{FT}, \cite{FHSSY}.  
We give here an account of the definition in \cite{BS} following \cite{Ng2}. 

Let us prepare some terminology. 
We say that an element $\nu=(\nu_1,\nu_2)\in \Z^2\backslash\{(0,0)\}$ 
is coprime if $\nu_1,\nu_2$ are coprime integers. 
We equip $\Z^2$ with the lexicographic ordering $>$: 
$(\nu_1,\nu_2)>(\nu_1',\nu_2')$ if $\nu_1>\nu'_1$, or $\nu_1=\nu'_1$ and $\nu_2>\nu_2'$.
For a triangle $T$ with vertices $\nu,\nu',\nu''\in \Z^2$,  
we write $\mathrm{middle}(T)=\nu'$ if $\nu>\nu'>\nu''$. We say that 
$T$ is a quasi-empty triangle if there are 
no lattice points in its interior and on at least one of its edges.

The elliptic Hall algebra $\A$ is generated by 
elements $\bp_\nu$ ($\nu\in\Z^2\backslash\{(0,0)\}$)
and central elements  $\bc_{\nu}$ ($\nu\in \Z^2$) 
satisfying $\bc_0=1$, $\bc_{\nu+\nu'}=\bc_\nu\bc_{\nu'}$. 
The defining relations read as follows. 
\begin{align*}
&[\bp_{r\nu},\bp_{s\nu}]
=\frac{r}{\kappa_r}\delta_{r+s,0}\left(\bc_{r\nu}-\bc_{-r\nu}\right)
\quad \text{ if $\nu$ is coprime and $r,s\in\Z\backslash\{0\}$},
\\
&[\bp_\nu,\bp_{\nu'}]=\frac{\bc_{M}}{\kappa_1} \bh_{\nu+\nu'}
\quad \text{if $-\nu',0,\nu$ form a quasi-empty 
triangle $T$ oriented clockwise}\,.
\end{align*}
Here $M=\mathrm{middle}(T)$, and the symbol $\bh_{\nu}$ is 
defined by setting 
\begin{align*}
\sum_{n\ge0}\bh_{n \nu}z^{-n}=
\exp\Bigl(-\sum_{r\ge1} \kappa_r \bp_{r\nu}\frac{z^{-r}}{r}\Bigr)
\end{align*}
for all coprime $\nu$. 

The generators $\bp_{(\nu_1,\nu_2)}$ correspond to the 
elements $x^{\nu_1}y^{\nu_2}$ of the Lie algebra. 
The following PBW type basis is known. 
For a non-zero vector $\nu=(r\cos\theta,r\sin\theta)$ on the plane,
we define its argument by $\arg \nu=\theta$, where 
$r>0$ and $-\pi<\theta\le\pi$. 
We say that a monomial $\bc_\nu \bp_{\nu^{(1)}}\cdots \bp_{\nu^{(N)}}\in\A$ 
($\nu\in\Z^2$, $\nu^{(i)}=(\nu^{(i)}_1,\nu^{(i)}_2)\in\Z^2\backslash\{(0,0)\}$) 
is normal-ordered if 
$\pi\ge \arg \nu^{(1)}\ge\arg \nu^{(2)}\ge\cdots\ge\arg \nu^{(N)}>-\pi$, 
and if $\arg \nu^{(i)}=\arg \nu^{(i+1)}$ then 
$\nu^{(i)}_1\ge \nu^{(i+1)}_1$. 

\begin{thm}\cite{BS}\label{thm:pbw}
The set of all normal-ordered monomials is a basis of $\A$. 
\end{thm}

The following result tells how the algebras $\E$ and $\A$ are related.
Let 
$\E'=\langle e_n,f_n (n\in\Z), h_r (r\in\Z\backslash\{0\}), C,C^\perp\rangle$  
be the subalgebra of $\E$ obtained by `dropping' $D,D^\perp$. 
\begin{thm}\label{A=E}\cite{S}
There is an isomorphism of algebras $\E'\overset{\sim}{\to}\A$ such that 
\begin{align*}
&e_n\mapsto \bp_{(1,n)}\,,\ f_n\mapsto \bp_{(-1,n)}\,, \ 
C^{\pm r}h_{\pm r}\mapsto -\frac{1}{r}\bp_{(0,\pm r)}\,,\\
&C\mapsto \bc_{(0,1)}\,,\ C^\perp\mapsto \bc_{(-1,0)}\,,
\end{align*}
where $n\in \Z$, $r>0$.
\end{thm}

Hereafter we identify $\E'$ with $\A$ 
by the isomorphism stated in Theorem \ref{A=E}. 

The presentation of $\A$ clarifies some properties which are not 
obvious in that of $\E'$. 
Among other things, the natural action of $SL(2,\Z)$ on $\Z^2$
lifts to an action of the universal cover $\widetilde{SL}(2,\Z)$
on $\A$ by automorphisms, see \cite{BS,Ng2}. 
We shall use a particular automorphism  $\theta$ of $\E$ 
of order four \cite{BS,M}
\begin{align}
\theta&:e_0\mapsto h_{-1},\ 
h_{-1}\mapsto f_0,\ f_0\mapsto h_1,\ h_1\mapsto e_0,
\label{theta}\\
&\Cp\mapsto C\,,\ C\mapsto (\Cp)^{-1}\,,\
D^\perp \mapsto D\,,\ D\mapsto (D^\perp)^{-1}\,,
\nn
\end{align}
which corresponds to rotating the lattice clockwise by $90$ degrees.  
Its square is an involutive automorphism 
\begin{align}
\theta^2&:e_n\mapsto f_{-n},\ f_n\mapsto e_{-n},\ h_r\mapsto h_{-r}, 
\ x\mapsto x^{-1} \quad (x=C,\Cp,D,D^\perp).
\label{th2}
\end{align}

Quite generally, we write $x^\perp=\theta^{-1}(x)$ 
for an element $x\in \E$.  In this notation
\begin{align*}
&\ep_0=h_1,\quad \fp_0=h_{-1},\quad \hp_1=f_0,\quad \hp_{-1}=e_0\,,\\
&\ep_{-1}=e_1C^{-1},\ \ep_1=f_1C^\perp,\  
\fp_1=f_{-1}C,\ \fp_{-1}=e_{-1}(C^\perp)^{-1}\,.
\end{align*}
We refer to 
\begin{align}
e_n\,,\quad f_n\,,\quad h_r\,,
\quad
C\,,\ C^\perp\,,\ D\,, \ D^\perp\,
\label{horizontal}
\end{align}
as {\it horizontal generators}, and  
\begin{align}
e^\perp_n\,,\quad f^\perp_n\,,\quad h^\perp_r\,,
\quad
C\,,\ C^\perp\,,\ D\,, \ D^\perp\,
\label{vertical}
\end{align}
as {\it vertical generators}, see Fig. \ref{fig2}.

\begin{figure}[t]
\setlength{\unitlength}{.8mm}
\begin{picture}(120,120)(-40,-60)
\put(6,55){$\vdots$}\put(26,55){$\vdots$}\put(46,55){$\vdots$}
\put(17,58){\line(0,-1){65}}
\put(17,12){\line(1,0){20}}
\put(17,-7){\line(1,0){20}}
\put(37,12){\line(0,-1){67}}
\put(-56,54){\vector(0,-1){5}}
\put(-56,54){\vector(1,0){5}}
\put(-49,53){ $\hdeg$}
\put(-56,43){$\pdeg$}
\put(-16,54){\oval(10,8)}
\put(-18,52){$\Bb^\perp$}
\put(66,54){\oval(10,8)}
\put(64,52){$\Bb$}
\put(66,-52){\oval(10,8)}
\put(64,-54){$\B^\perp$}
\put(-16,-52){\oval(10,8)}
\put(-18,-54){$\B$}
\put(5,40){$\fp_2$}\put(25,40){$\hp_2$}
\put(45,40){$\ep_2$}
\put(-40,20){$\cdots$}\put(-20,20){$f_{-2}$}
\put(-2,20){$f_{-1}$, }\put(6,20){\ $\fp_1$}
\put(20,20){$f_0$, }\put(25,20){\ $\hp_1$}
\put(40,20){$f_1$, }\put(45,20){\ $\ep_1$}
\put(60,20){$f_2$}\put(80,20){$\cdots$}
\put(-40,00){$\cdots$}\put(-20,00){$h_{-2}$}
\put(-2,00){$h_{-1}$}\put(6,00){$, \fp_{0}$}
\put(24,00){$\bullet$}
\put(40,00){$h_1$}\put(45,00){$, \ep_0$}
\put(60,00){$h_2$}\put(80,00){$\cdots$}
\put(-40,-10){\line(1,0){75}}
\put(35,-10){\line(0,1){19}}
\put(15,-10){\line(0,1){19}}
\put(15,9){\line(1,0){65}}
\put(-40,-20){$\cdots$}\put(-20,-20){$e_{-2}$}
\put(-2,-20){$e_{-1}$}\put(6,-20){$, \fp_{-1}$}
\put(20,-20){$e_0$}\put(25,-20){$, \hp_{-1}$}
\put(40,-20){$e_1$}\put(45,-20){$, \ep_{-1}$}
\put(60,-20){$e_2$}\put(80,-20){$\cdots$}
\put(5,-40){$\fp_{-2}$}
\put(25,-40){$\hp_{-2}$}
\put(45,-40){$\ep_{-2}$}
\put(7,-55){$\vdots$}\put(27,-55){$\vdots$}\put(47,-55){$\vdots$}
\end{picture}
\label{fig2}
\caption{Horizontal/vertical generators and 
Borel subalgebras 
$\B$ (lower half), $\Bb$ (upper half), 
$\B^\perp$ (right half), $\Bb^\perp$ (left half). 
The elements $C, C^\perp,D,D^\perp$ 
placed at the center $\bullet$ are common to all these subalgebras.}
\end{figure}
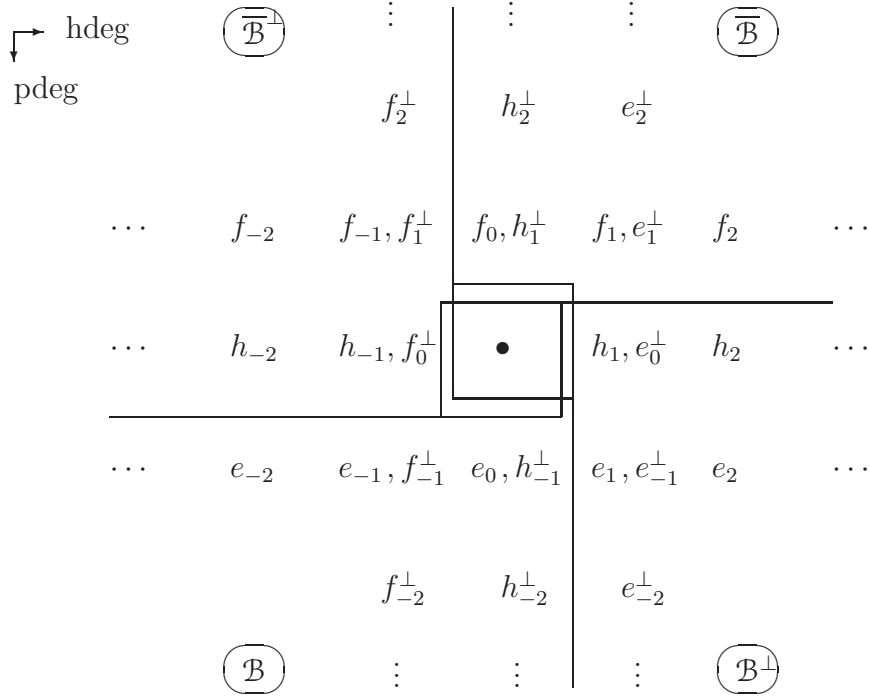

\subsection{Hopf algebra structures}

Algebra $\E$ is endowed with a topological 
Hopf algebra structure $(\E,\Delta,\varepsilon,S)$
defined in terms of the horizontal generators. 
For $x=C,\Cp,D,D^\perp$, we set
$\Delta\, x=x\otimes x$, $\varepsilon(x)=1$, $S(x)=x^{-1}$.
For the remaining generators we define 
\begin{align}
&\Delta e_n=\sum_{j\ge0}e_{n-j}\otimes 
\psi^{+}_j C^n+1\otimes e_n\,,
\label{copro1}\\
&\Delta f_n=f_n\otimes 1+
\sum_{j\ge0}\psi^{-}_{-j}C^n\otimes f_{n+j}\,,
\label{copro2}\\
&\Delta h_{r}=h_{r}\otimes 1+C^{-r}\otimes h_{r}\,,
\label{copro3}\\
&\Delta h_{-r}=h_{-r}\otimes C^r+1\otimes h_{-r}\,,
\label{copro4}\\
&\varepsilon(e_n)=\varepsilon(f_n)=\varepsilon(h_{\pm r})=0\,,
\label{counit}
\\
&S(e_n)=-\sum_{j\ge0}C^{-n+j}e_{n-j}\widetilde{\psi}^+_j\,,
\\
&S(f_n)=-\sum_{j\ge0}\widetilde{\psi}^-_{-j}C^{-n-j}f_{n+j}\,,
\\
&S(h_{\pm r})=-h_{\pm r}C^{\pm r}\,. 
\end{align}
Here we set 
 $\psi^{\pm}(z)^{-1}=\sum_{j\ge0}\widetilde{\psi}^{\pm}_{\pm j}z^{\mp j}$.

Replacing the horizontal generators with the corresponding vertical generators, 
we obtain another Hopf algebra structure $(\E, \Delta^\perp,\varepsilon^\perp,S^\perp)$.  
Both structures will play a role in the sequel. 

\subsection{Borel subalgebras}\label{subsec:Borel}
The lattice $\Z^2$ is partitioned as
$\Z^2=\{(0,0)\}\cup Z\cup \overline{Z}$, 
where 
$Z=\{\nu=(\nu_1,\nu_2)\in\Z^2\mid \text{$\nu_1>0$ or $\nu_1=0,\nu_2>0$}\}$, 
$\overline{Z}=-Z$. 
The corresponding elements $\bp_\nu$ generate subalgebras of $\E$, 
\begin{align*}
&\B=\langle \bp_\nu\ (\nu\in Z),\ C,\Cp, D, D^\perp \rangle \,,
\\
&\Bb=\langle \bp_\nu\ (\nu\in \overline{Z}),\ C,\Cp, D, D^\perp \rangle\,,
\end{align*}
which we call the Borel and the opposite Borel subalgebras, respectively,  
see Fig. \ref{fig2}. 
Both $(\B,\Delta,\varepsilon,S)$, $(\Bb,\Delta,\varepsilon,S)$ are  
Hopf subalgebras of $(\E,\Delta,\varepsilon,S)$. 
Using the presentation of the elliptic Hall algebra, we see that
algebra $\B$ is generated by 
$e_n$ ($n\in\Z$), $h_r$ ($r>0$), $C,\Cp, D, D^\perp$, 
and $\Bb$ by $f_n$ ($n\in\Z$), $h_{-r}$ ($r>0$), $C,\Cp, D, D^\perp$.

Quite generally, a bialgebra pairing on a bialgebra $A$
is a symmetric non-degenerate bilinear form $(~,~):A\times A\to \C$ 
with the properties
\begin{align*}
&(a,b_1b_2)=(\Delta(a), b_1\otimes b_2),\quad
(a,1)=\varepsilon(a)\,
\end{align*}
for any $a,b_1,b_2\in A$. 
With each such pair $(A,(~,~))$, there is an associated bialgebra
$DA$ called the Drinfeld double of $A$. 
As a vector space $DA=A\otimes A^{\mathrm{op}}$,  
where $A^{\mathrm{op}}$ is a copy of $A$ endowed with the opposite coalgebra
structure. Moreover 
$A^+=A\otimes 1$ and $A^-=1\otimes A^{\mathrm{op}}$ are sub bialgebras of $DA$, 
and the commutation relation
\begin{align*}
\sum (a_{(2)},b_{(1)})\, a^-_{(1)}b^+_{(2)}
=\sum (b_{(2)},a_{(1)})\, b^+_{(1)}a^-_{(2)}
\end{align*} 
is imposed for $a,b\in A$. Here $a^+=a\otimes 1$, $a^-=1\otimes a$, and 
we use the Sweedler notation 
$\Delta(a)=\sum a_{(1)}\otimes a_{(2)}$ for the coproduct.  

In the present case, we take $A$ to be the bialgebra
$(\B,\Delta,\varepsilon)$. 
It has a bialgebra pairing such that the non-trivial pairings of the generators are given by
\begin{align*}
&(e_m,e_n)=\frac{1}{\kappa_1}\delta_{m,n}\,,\quad
(h_r,h_{s})=\frac{1}{r\kappa_r}\delta_{r,s}\,,\\
&(C,D)=(C^\perp,D^\perp)=q^{-1}\,.
\end{align*}
This pairing respects bidegrees in the sense that $(a,b)=0$ unless $\deg a=\deg b$.
We identify $A^{\mathrm{op}}$ with $\Bb$ via the involution 
$\theta^2:\B\to \Bb$ given in \eqref{th2}. 
The Drinfeld double of $\B$ is then identified with $\B\otimes \Bb$. 
Its quotient by the relation $x\otimes 1=1\otimes x$ 
($x=C,C^\perp,D,D^\perp$) is isomorphic to the algebra $\E$ \cite{BS}.  

We can equally well use 
the `vertical' version of Borel and opposite Borel subalgebras
\begin{align*}
\B^\perp=\theta^{-1}(\B)\,, \quad  \Bb^\perp=\theta(\B)\,.
\end{align*}
Algebra $\E$ is also a quotient of the Drinfeld double of $(\B^\perp,\Delta^\perp,\varepsilon^\perp)$.

\section{Modules}\label{sec:modules}

This section contains preliminary material on representations of $\E$ and $\B^\perp$.  
We shall consider only modules on which $C$ acts by $1$. 
From here until the end of the paper, we drop $D$ from the algebra and 
use the same letters $\E$, $\B^\perp$, etc., to denote the respective 
quotient by $C-1$.  

\subsection{Category $\cO_\E$}

Let $V$ be an $\E$ module. For $a\in\C^\times$ and $n\in \Z$, we set
\begin{align*}
&V_{(a)}=\{v\in V\mid  \psi^+_0v=a v\}\,,\quad 
V_{n}=\{v\in V\mid D^\perp v=q^n v\}\,,
\end{align*}
and $V_{(a,n)}=V_{(a)}\cap V_n$. 
If $V_{(a,n)}\neq 0$ then we call $(a,n)$ a weight of $V$.  
We denote by $\wt(V)$ the set of weights of $V$. 
For a vector $v\in V_n$ we write $\pdeg v=n$. 

We consider a full subcategory $\cO_\E$ of the category of 
all $\E$ modules. 
We say that $V$ is an object of $\cO_\E$ if the following 
conditions are satisfied. 
\begin{enumerate}
\item There exists a finite subset $A\subset\C^\times$ 
such that $V=\bigoplus_{(a,n)\in A\times\Z}V_{(a,n)}$,
\item $\dim V_{(a,n)}<\infty$ for all $a,n$,
\item $V_{n}=0$ if $n\gg 0$.
\end{enumerate}
Thanks to (iii), the coproduct $\Delta^\perp$ is
well-defined on the tensor product of $V,W\in \Ob\cO_\E$. 
We write this tensor product module as $V\otimes_{\Delta^\perp} W$. 
Category $\cO_\E$ is a monoidal category 
with respect to $\otimes_{\Delta^\perp}$. 

Let $\rg_\E$ be the set of all rational functions 
$\Psi(z)\in \C(z)$ which are
regular at $z=0,\infty$ and satisfy $\Psi(0)\Psi(\infty)=1$. 
Let $\Psi^\pm(z)=\sum_{j\ge 0}\Psi^\pm_{\pm j}z^{\mp j}$ 
be the expansion of $\Psi(z)$ at $z^{\pm 1}=\infty$.

We say that an $\E$ module $V$ is 
a highest $\ell$-weight module with highest $\ell$-weight $\Psi\in\rg_\E$ 
if it is generated by a non-zero vector $v_0$ such that 
\begin{align*}
\bp_\nu v_0=0\ (\nu_1>0),\quad \psi^\pm(z)v_0=\Psi^\pm(z)v_0\,, 
\quad D^\perp v_0=v_0\,.
\end{align*}
The unique simple module with highest $\ell$-weight $\Psi$ 
belongs to $\cO_\E$. We  denote it by $L(\Psi)$. 
Modulo a shift of grading by $D^\perp$, 
any simple object of category $\cO_\E$ is 
of the form $L(\Psi)$ for some $\Psi\in\rg_\E$, see \cite{M}. 
\medskip

We are interested in the Grothendieck ring 
$\Rep_0\, \E$ of category $\cO_\E$. However,
many natural representations of $\E$ do not possess
finite composition series. 
In order to address this issue we pass to an 
appropriate completion of 
$\Rep_0\, \E$. 
Namely,
for $n\in \Z$, let $F^n$ denote the additive subgroup of 
$\Rep_0\, \E$ 
spanned by $[V]$, $V\in \cO_\E$, such that $V_m=0$ for $m>n$. 
This gives a filtration 
$\Rep_0\, \E=\cup_{n\in \Z}F^n$, 
$\cdots\supset F^1\supset F^0\supset F^{-1}\supset\cdots$, 
which satisfies $F^mF^n\subset F^{m+n}$, $F^m+F^n\subset F^{\min(m,n)}$. 
We set 
\begin{align}\label{complete}
\Rep\,\E ={\varprojlim_{n\to-\infty}} \ (\Rep_0\, \E)/F^n. \qquad 
\end{align}

\subsection{Examples of highest $\ell$-weight modules} 
To write the examples of highest $\ell$-weight modules 
we need the notation for partitions and plane partitions.

\medskip

A partition is a sequence of non-negative 
integers $\la=(\la_1,\la_2,\cdots)$ such that $\la_i\ge\la_{i+1}$ for all 
$i\ge1$, and $\la_i=0$ for $i$ large enough.
We write $\emptyset=(0,0,0,\dots)$.
The set of all partitions is denoted by $\cP$. 
We identify a partition $\la\in\cP$ with the set of points $(i,j)\in\Z^2$
such that $1\le j\le \la_i$, $i\ge1$. 
For $\la\in\cP$, we set $|\la|=\sum_{j\ge1}\la_j$  
and $\ell(\la)=\max\{j\ge1\mid \lambda_j>0\}$ ($\lambda\neq\emptyset$), 
$\ell(\emptyset)=0$. 

For $\la\in\cP$, $i\in \Z_{>0}$, we create a new sequence of non-negative integers $\la+{\bf 1}_i$ by the rule
$(\la+{\bf 1}_i)_j=\la_j+\delta_{i,j}$. We also use the cut partitions:
\begin{align*}
(\la)_{\le i-1}=(\la_1,\cdots,\la_{i-1},0,0,\cdots), 
\qquad (\la)_{\ge i+1}=(\la_{i+1},\cdots,\la_{\ell(\la)},0,0,\cdots)\,.
\end{align*}

\medskip 

A plane partition is a sequence of partitions
$\bla=(\la^{(1)},\la^{(2)},\cdots)$, $\la^{(k)}\in \cP$, 
such that $\la^{(k)}\supset\la^{(k+1)}$ for all $k$ 
and $\la^{(k)}=\emptyset$ for $k\gg0$. We denote the set of all plane partitions by $\bs\cP$ and for $\bs \la\in\bs\cP$ we set $|\bs\la|=\sum_i|\la^{(i)}|$.

We identify 
$\bla\in\bs\cP$ with the set of points $(i,j,k)\in\Z^3$ such that 
$1\le j\le \la^{(k)}_i$, $i,k\ge1$. 

For a plane partition $\bla=\bigl(\la^{(1)},\cdots,\la^{(N)}\bigr)$ we set
\begin{align*}
(\bla)_{k,i}=\bigl(\la^{(1)},\cdots,\la^{(k-1)},(\la^{(k)})_{\le i-1},
\emptyset,\emptyset,\cdots\bigr)\,,
\quad
(\bla)^{k,i}=\bigl((\la^{(k)})_{\ge i+1}, \la^{(k+1)},\cdots,\la^{(N)},
\emptyset,\emptyset,\cdots\bigr)\,,
\end{align*}
and
\begin{align*}
\bla+{\bf 1}^{(k)}_i=\bigl(\la^{(1)},\cdots,\la^{(k)}+{\bf 1}_i,\cdots,
\la^{(N)},\emptyset,\emptyset,\cdots\bigr)\,.
\end{align*}

For a node $\square=(i,j,k)\in\Z^3$ we write $q^{\square}=q_3^iq_1^jq_2^k$. We also use
\begin{align}\label{A}
\af\bigl(a\bigr)=\prod_{s=1}^3\frac{1-q_s^{-1}a}{1-q_s a}\,.
\end{align}
\medskip

The most basic
example of a highest $\ell$-weight module 
is the Macmahon module $\cM(u,K)$ \cite{FJMM1}.
The Macmahon module $\cM(u,K)$ has a basis $\{\ket\bla\}$ 
labeled by the set of all plane partitions with 
the action of $\E$ given by the following explicit formulas. Namely, we have
\begin{align}
\bra{\bla}\psi^{\pm}(z) \ket{\bla}=
K^{1/2}\frac{1-K^{-1}u/z}{1-u/z}
\prod_{\square \in \bla}\af\bigl(q^{-\square}u/z\bigr)^{-1}\,.
\label{Mac-psi}
\end{align}
If $\bla+{\bf 1}^{(k)}_i$ is a plane partition, then we have
\begin{align}&\bra{\bla+{\bf 1}^{(k)}_i}f(z) \ket{\bla}
=\frac{K^{1/2}}{1-q_1^{-1}}
\frac{1-q_2^{-k}u/z}{1-u/z}
\frac{1-q_3^{-i+1}q_2^{-k+1}u/z}
{1-q_3^{-i+1}q_2^{-k}u/z}
\prod_{\square \in (\bla)_{k,i}}\af\bigl(q^{-\square}u/z\bigr)^{-1}
\label{Mac-f}
\\
&\hspace{300.pt} \times\delta\bigl(q_3^{-i}q_1^{-\la^{(k)}_i-1}
q_2^{-k}u/z\bigr)\,,
\nn\\
&\bra{\bla}e(z) \ket{\bla+{\bf 1}^{(k)}_i}
=\frac{-1}{1-q_1}
\frac{1-K^{-1} u/z}{1-q_2^{-k}u/z}\frac{1-q_3^{-i}q_2^{-k}u/z}{1-q_3^{-i}q_2^{-k+1}u/z}
\prod_{\square \in (\bla)^{k,i}}\af\bigl(q^{-\square}u/z\bigr)^{-1}\ 
\label{Mac-e}
\\
&\hspace{300.pt} \times
\delta\bigl(q_3^{-i}q_1^{-\la^{(k)}_i-1}q_2^{-k}u/z\bigr)\,.
\nn
\end{align}
All other matrix coefficients are zero.

For generic  $K\in\C^\times$ (i.e. $K\neq q_3^mq_1^n$ for $m,n\in\Z$) 
the Macmahon module is irreducible,
\begin{align*}
\cM(u,K)=L\left(\Psi_{\cM(u,K)}\right)\,,\quad 
\Psi_{\cM(u,K)}(z)=K^{1/2}\frac{1-K^{-1}u/z}{1-u/z}\,.
\end{align*}
Clearly, the Macmahon modules $\cM(u,K)$, $u,K\in\C^*$, 
topologically generate the ring $\Rep\,\E$.
\medskip

Another important example of a highest $\ell$-weight module is the Fock module
\begin{align}
\F(u)=L\left(\Psi_{\F(u)}\right)\,,\quad 
\Psi_{\F(u)}(z)=q\frac{1-q_2^{-1}u/z}{1-u/z}\,.
\label{Fock}
\end{align}
The Fock module $\F(u)$ is the smallest highest $\ell$-weight $\E$ module.
It has a basis labeled by partitions. The action of $\E$ is obtained from formulas \eqref{Mac-psi},\eqref{Mac-f}, \eqref{Mac-e} for the action in the Macmahon module by specializing $K=q_2$ and forgetting all plane partitions $\bla=(\la^{(1)},\la^{(2)},\cdots)$ with $\la^{(2)}\neq \emptyset$.
\medskip

The Fock module can be described also 
in terms of the vertical generators.
The generators $\hp_r$ act as a Heisenberg algebra on $\F(u)$,  
\begin{align}
[\hp_r,\hp_s]=\frac{q^r-q^{-r}}{r\kappa_r}\delta_{r+s,0}
\quad (r,s\in\Z\backslash\{0\})
\,. 
\label{Heis}
\end{align}
For $r>0$, $h^\perp_r$'s act as creation operators
and $h^\perp_{-r}$'s as annihilation operators. 
Module $\F(u)$ is irreducible over this Heisenberg algebra.
The generators $\ep(z),\fp(z)$ act by vertex operators, 
\begin{align}
&\ep(z)=\frac{1-q^{-1}_2}{\kappa_1}\,
u \,
\exp\left(-\sum_{r=1}^\infty\frac{q^r\kappa_r}{1-q_2^r}\,\hp_{r}z^{-r}\right)
\exp\left(-\sum_{r=1}^\infty \frac{q_2^r\kappa_r}{1-q_2^r}\,\hp_{-r} z^{r}\right)\,,
\label{VO1}\\
&\fp(z)= \frac{1-q_2}{\kappa_1} u^{-1}\,
\exp\left(\sum_{r=1}^\infty \frac{\kappa_r}{1-q_2^r}\,\hp_{r} z^{-r}\right)
\exp\left(\sum_{r=1}^\infty \frac{q^r\kappa_r}{1-q_2^r}\,\hp_{-r} z^{r}\right)\,.
\label{VO2}
\end{align}

\medskip

\subsection{Dual modules}

Along with $\cO_\E$, we consider also a category $\cO^\vee_\E$. It is 
defined similarly as $\cO_\E$, replacing condition (iii) with
\begin{enumerate}
\setcounter{enumi}{2}
\item$\hskip-.8mm^\vee$\quad  \text{$V_{n}=0$ if $n\ll 0$}.
\end{enumerate}
We define a lowest $\ell$-weight module with lowest $\ell$-weigth $\Psi\in\rg_\E$ to be 
a cyclic $\E$ module $V=\E v_0$ such that 
\begin{align*}
\bp_\nu v_0=0\ (\nu_1<0),\quad \psi^\pm(z)v_0=\Psi^\pm(z)v_0\,,\quad
D^\perp v_0=v_0\,. 
\end{align*}
The unique simple module with lowest $\ell$-weight $\Psi\in\rg_\E$ belongs to $\cO^\vee_\E$ and 
is denoted by $L^{\vee}(\Psi)$. 

If $V\in\Ob\cO_\E$ and $W\in\Ob\cO^\vee_\E$, then 
the tensor product $V\otimes_{\Delta^\perp} W$ is a well defined $\E$ module.
(Note however that the tensor product in the opposite order $W\otimes_{\Delta^\perp} V$ is ill defined.) 
For an element $x\in \E$, the antipode $S^\perp(x)$ is well-defined on 
$W\in \Ob\cO^\vee_\E$.  
This allows us to define an $\E$ module structure on the graded dual space $W^*=\oplus_n W_n^*$ 
by setting $(xw^*)(w)=w^*\bigl(S^\perp(x)w\bigr)$, where $w^*\in W^*,\ w\in W,\ x\in \E$. 
The result is an object in $\cO_\E$, which we denote by $W^{*S^\perp}$.
Similarly, $(S^\perp)^{-1}(x)$ is well-defined on 
$V\in\Ob\cO_\E$ so that $V^{*(S^\perp)^{-1}}$ is defined. 
With the assignments 
$\Ob\cO^\vee_\E\to \Ob\cO_\E$, $W \mapsto W^{*S^\perp}$ and 
$\Ob\cO_\E\to \Ob\cO^\vee_\E$, $V \mapsto V^{*(S^\perp)^{-1}}$, 
we obtain contravariant functors which are inverse to each other. 

\begin{lem}\label{lem:duality}
For $\Psi\in\rg_\E$, we have
\begin{align}
L^{\vee}(\Psi)^{*S^\perp}=L(\Psi^{-1}),
\quad 
L(\Psi)^{*(S^{\perp})^{-1}}=L^{\vee}(\Psi^{-1}) \,.
\label{duality}
\end{align}
\end{lem}
\begin{proof}
It is enough to prove the first equality. 
Let $W=L^\vee(\Psi)$. 
Since $W^{*S^\perp}$ is simple, we can write $W^{*S^\perp}=L(\Phi)$ with some $\Phi\in\rg_\E$. 
Let $w_0$ be the lowest $\ell$-weight vector of $W$ 
and $w_0^*$ the highest $\ell$-weight vector of $W^{*S^\perp}$. 
There exists an $\E$ linear map
$W^{*S^\perp}\otimes_{\Delta^\perp} W\to \C$ onto the trivial module $\C$, 
sending $w_0^*\otimes w_0$ to $1$. 
On the other hand, by Lemma \ref{Delta-psi} we have 
$\Delta^\perp \psi^{\pm}(z)(w_0^*\otimes w_0)
=(\psi^\pm(z)w_0^*)\otimes (\psi^\pm(z)w_0)$. 
It follows that $1=\Phi(z)\Psi(z)$. 
\end{proof}
\medskip

\noindent{\it Remark.}\quad 
The Fock module $\F_u$ considered in our previous papers \cite{FFJMM1}, \cite{FJMM1}  
is the dual Fock module
\begin{align*}
\F^\vee(u)=L^\vee\left(q^{-1}\frac{1-q_2u/z}{1-u/z}\right)
=\bigl(\F(q_2u)\bigr)^{*(S^\perp)^{-1}}\,.
\end{align*}
\qed

\subsection{Category $\cO_{\B^\perp}$}
Consider now modules over the Borel subalgebra $\B^\perp$. 
Category $\cO_{\B^\perp}$ is defined by the same conditions (i)--(iii) as for category $\cO_\E$,  
with $\B^\perp$ modules in place of $\E$ modules. We denote by $\Rep\,\B^\perp$ the same completion of the
Grothendieck ring of $\cO_{\B^\perp}$ as for  $\Rep\,\E$, see \eqref{complete}.

In order to define highest $\ell$-weight modules, we introduce the following subalgebras of $\B^\perp$.
\begin{align}
&\B^\perp_+=\langle \bp_\nu\ (\nu_1>0,\ \nu_2>0) \rangle \,, \qquad \B^\perp_-=\langle \bp_\nu\ (\nu_1<0,\ \nu_2\geq 0) \rangle 
\label{Bpm}\\
&\B^\perp_0=\langle h_r\ (r>0),\ C,\Cp, D, D^\perp \rangle \,,
\label{B0}
\end{align}
We have $\B^\perp_0\B^\perp_+=\B^\perp\cap\B$, $\B^\perp_-=\B^\perp\cap\Bb$, and 
the triangular decomposition holds:
\begin{align*}
\B^\perp\simeq \B^\perp_-\otimes \B^\perp_0\otimes \B^\perp_+\,. 
\end{align*}

Let $\rg_{\B^\perp}$ denote the set of all rational functions
$\Psi(z)\in\C(z)$ which are regular and non-zero at $z=\infty$.
Let $\Psi^+(z)=\sum_{m\ge0}\Psi^+_mz^{-m}\in\C[[z^{-1}]]$ 
be its expansion at $z=\infty$.

A cyclic $\B^\perp$ module $V=\B^\perp v_0$ is a highest $\ell$-weight module 
of highest $\ell$-weight $\Psi\in \rg_{\B^\perp}$ if
\begin{align*}
\B^\perp_+v_0=\C v_0\,,\quad \psi^+(z)v_0=\Psi^+(z)v_0\,,
\quad D^\perp v_0=v_0\,.
\end{align*}
Similarly $V=\B^\perp v_0$ is a lowest $\ell$-weight module with lowest $\ell$-weight $\Psi\in\rg_{\B^\perp}$ if 
\begin{align*}
\B^\perp_- v_0=\C v_0,\quad \psi^+(z)v_0=\Psi^+(z)v_0\,,
\quad D^\perp v_0=v_0\,. 
\end{align*}

The unique simple module with highest (resp. lowest) $\ell$-weight 
$\Psi$ is denoted by
$L(\Psi)$ (resp. $L^\vee(\Psi)$).  
The duality relation 
\eqref{duality} 
holds true for them as well. 

We have the following result analogous 
to the one for $\E$ modules \cite{M}:
\begin{lem}
Let $V$ be a simple $\B^\perp$ module. Then $V\in\Ob\cO_{\B^\perp}$ if and only if $V=L(\Psi)$ for some $\Psi\in \rg_{\B^\perp}$. 
\qed
\end{lem}
\medskip 

Consider highest $\ell$-weight $\B^\perp$ modules 
\begin{align*}
M^+(u)=L\bigl(1-u/z\bigr)\,,
\quad 
M^-(u)=L\Bigl(\frac{1}{1-u/z}\Bigr)\,.
\end{align*}
Clearly, modules $M^+(u)$ and $M^-(u)$, $u\in\C^\times$, 
topologically generate the ring $\Rep\, \B^\perp$.

We call $M^+(u)$ (resp. $M^-(u)$) the 
{\it positive (resp. negative) fundamental module}.
\medskip

We have the restriction functor $Res:\cO_\E\to \cO_{\B^{\perp}}$  
which sends an $\E$ module to its restriction to $\B^\perp$.  
\begin{lem}\label{lem:simple-to-simple}
Functor $Res:\cO_\E\to \cO_{\B^{\perp}}$ sends 
simple objects to simple objects.
Hence the notation $L(\Psi)$ has an unambiguous meaning for $\Psi\in\rg_\E$. 
\end{lem}
\begin{proof}
The following proof is adapted from the one
given in Proposition 3.5 of \cite{HJ}.

Let $V=\oplus_{n\in\Z}V_n$ be an object in $\cO_\E$. 
Fix $n$, and consider the set of operators $e_m:V_n\to V_{n+1}$ ($m>0$).
Since $\Hom_\C(V_n,V_{n+1})$ is finite dimensional, 
we have a linear relation  $\sum_{j=a}^b c_j e_j|_{V_n}=0$ 
where $c_j\in\C$ ($0<a\le j\le b$) and $c_ac_b\neq0$. 
Taking commutators with 
$h_{\pm1}$ we obtain $\sum_{j=a}^b c_j e_{j+r}|_{V_n}=0$ for all $r\in \Z$. 
It follows that operators 
$e_k|_{V_n}$ ($k\in\Z$) 
belong to the linear span of $\{e_m|_{V_n}\}_{m>0}$.
By the same argument, operators $f_k|_{V_n}$ ($k\in\Z$) 
belong to the linear span of $\{f_m|_{V_n}\}_{m\ge0}$. 

It is clear now that any singular vector of $V$ with respect to $\B^\perp$
is also singular with respect to $\E$. It is also clear that
if $V$ is cyclic with respect to $\E$ then it is cyclic with respect to 
$\B^\perp$. The assertion of the lemma follows from these.
\end{proof}

Since $\rg_{\B^\perp}\neq \rg_\E$, there are $\B^\perp$ modules which are not obtained by
restricting $\E$ modules. For example, positive and negative fundamental modules are not restrictions of $\E$ modules.

\medskip

We use the linear maps
\begin{align}
s_u:L(\Psi(z/u))\to L(\Psi(z))\,,
\quad 
\tau_a:L(a\Psi(z))\to L(\Psi(z))\,,
\label{s-tau}
\end{align}
where $\Psi\in\rg_{\B^\perp}$, $u,a\in \C^\times$, given as follows. 

Both maps send highest $\ell$-weight vectors to highest 
$\ell$-weight vectors.
In addition, $s_u$ satisfies 
\begin{align*}&s_u\circ x= 
u^{\hdeg x}
x \circ s_u\quad (\text{for homogeneous $x\in \B^\perp$}),
\end{align*}
and the map $\tau_a$ satisfies
\begin{align*}
&\tau_a\circ e_n^\perp=e_n^\perp \circ \tau_a\quad (n\in\Z),\\
&\tau_a\circ h^\perp_r=a^r h_r^\perp \circ \tau_a\quad (r>0),\\
&\tau_a\circ C^\perp =a^{-1} C^\perp \circ \tau_a\,.
\end{align*}
We have $\tau_a\circ e_n =e_n\circ \tau_a$ $(n>0)$,
$\tau_a\circ f_n =a f_n\circ \tau_a$ $(n\ge0)$
and $\tau_a\circ \psi^+(z)=a \psi^+(z) \circ \tau_a$. 

\begin{lem} 
For any $\Psi\in\rg_{\B^\perp}$, the conditions above define maps 
$s_u$ and $\tau_a$.
Moreover, $s_u$ and $\tau_a$ are linear isomorphisms.
\qed
\end{lem}

\subsection{The $q$-characters}

Let $\Z[\C^\times]$ be the group ring 
of the multiplicative group $\C^\times$. 
We use the letter $x_a$ to denote the element $a\in\C^\times$, 
so that $x_{ab}=x_ax_b$. 
For an object $V$ of $\cO_{\B^\perp}$, we define its character by
\begin{align*}
\chib(V)=\sum_{(a,n)\in \wt(V)}\dim V_{(a,n)}\ x_a t^n\,. 
\end{align*}
The character $\chib(V)$
belongs to the ring 
$\cX_0=\Z[\{x_a\}_{a\in\C^\times},t][[t^{-1}]]$, 
and the map $\chib:\Rep\,\B^\perp\to \cX_0$ gives a ring homomorphism. 

The $q$-character is a refined notion of the character. It
is defined to be the generating function of the generalized  
eigenvalues of the commuting family of operators $\psi^+(z)$.  
In order to give the formal definition we prepare some notation. 

For $\Psi\in\rg_{\B^\perp}$ we set $V_{(\Psi,n)}=V_\Psi\cap V_{(a,n)}$, 
where $a=\Psi^+_0$ and 
\begin{align*}
V_\Psi&=\Bigl\{v\in V\mid \text{$\exists N$ such that 
$(\psi^+_{j}-\Psi^+_{j})^Nv=0$ 
for all $j\ge1$} 
\Bigr\}. 
\end{align*}
If $V_{(\Psi,n)}\neq 0$, we say that $(\Psi,n)$ is an $\ell$-weight of $V$. 
The set of $\ell$-weights of $V$ is denoted by $\elwt(V)$. 
We say that a subspace $W\subset V$ is $\ell$-weighted if $W=\oplus_{(\Psi,n)\in\elwt(V)} (V_{(\Psi,n)}\cap W)$.

We introduce indeterminates $X_a$ labeled by $a\in\C^\times$, 
and set, cf. \eqref{A},
\begin{align}
A_a=\prod_{s=1}^3\frac{X_{q^{-1}_s a}}{X_{q_s a}}\,.
\label{Aa}
\end{align}
Note that $\{A_a\}_{a\in\C^\times}$ are algebraically independent.

For an element $\Psi\in\rg_{\B^\perp}$, we define $\bm(\Psi)\in\cX$ as follows.
\begin{align*}
\bm(\Psi)=x_a\,\frac{\prod_{i=1}^l X_{a_i}}{\prod_{j=1}^m X_{b_j}} 
\quad
\text{if 
$\Psi(z)=a\,\frac{\prod_{i=1}^l(1-a_i/z)}{\prod_{j=1}^m(1-b_j/z)}$}\,.
\end{align*}
Clearly $\bm$ is a group isomorphism from the multiplicative group 
$\rg_{\B^\perp}$ onto the group of monomials in $\{x_a,X_a^{\pm1}\}_{a\in\C^\times}$. 

We define the $q$-character of an object $V$ of $\cO_{\B^\perp}$ 
by setting 
\begin{align*}
\chi_q(V)=\sum_{(\Psi,n)\in \elwt(V)} \dim V_{(\Psi,n)}\cdot
\bm(\Psi)\,t^n\,.
\end{align*}
It turns out that $\chi_q(V)$ belongs to the ring 
\begin{align*}
\cX=\Z[\{X^{\pm1}_a,x_a\}_{a\in\C^\times},t][[t^{-1}]]
\end{align*}
consisting of formal series in $t^{-1}$, whose coefficients
are polynomials in $X^{\pm1}_a,x_a$ ($a\in\C^\times$) and $t$. 
To see this it is enough to prove it in the case $V=L(\Psi)$
($\Psi\in\rg_{\B^\perp}$). This will be done in Lemma \ref{lem:q-char V}.

Under the ring homomorphism $\overline{\phantom{X}}:\cX\to\cX_0$ 
given by $\overline{X_a}=1$, $\overline{x_a}=x_a$, $\bar{t}=t$, 
the $q$ character specializes to the (ordinary) character,  
$\overline{\chi_q(V)}=\chib(V)$.

\medskip

\begin{prop}\label{inj}
The $q$ character map $\chi_q:\Rep\, \B^\perp\to \mathcal{X}$ is an injective ring homomorphism. 
In particular, $\Rep\, \B^\perp$ is a commutative ring.  
\end{prop}
\begin{proof}
The multiplicative property 
$\chi_q(V\otimes_{\Delta^\perp} W)=\chi_q(V)\chi_q(W)$ 
follows from Lemma \ref{Delta-psi}. 

To show the injectivity, let $[V_0]-[V_0']\in \Rep \B^\perp$ be an element
in the kernel of $\chi_q$ where $V_0,V_0'\in\Ob\cO_{\B^\perp}$. 
Choose a term $\mb(\Psi)t^n$ in $\chi_q(V_0)=\chi_q(V'_0)$ which has 
the highest power of $t$, and let 
$v_0\in V_0$ and $v_0'\in V_0'$ be eigenvectors of $\psi^+(z)$ 
corresponding to the eigenvalue $\Psi(z)$. 
Then the submodule of $V_0$ generated by $v_0$ 
has a quotient isomorphic to $L=L(\Psi)$. 
The same is true for $v_0'$ and $V'_0$. 
Hence we can write $[V_0]=[L]+[V]$, 
$[V'_0]=[L]+[V']$ with some $V,V'\in\Ob\cO_{\B^\perp}$. 
Repeating this procedure, 
we obtain a sequence of objects
 $W_k,V_k,V'_k\in \Ob\cO_{\B^\perp}$ ($k\ge0$)
such that 
$[V_{k-1}]=[W_k]+[V_{k}]$, $[V'_{k-1}]=[W_k]+[V_k']$, 
and that the highest power of $t$ in $\chi_q(V_k)=\chi_q(V_k')$
is less than that of $\chi_q(V_{k-1})$. 
It follows that $[V_0],[V'_0]$ are both equal to $[\oplus_{k}W_k]$. 
\end{proof}

\subsection{Examples of characters and $q$-characters}
We give a few characters and $q$-characters.

\medskip

The character and the $q$-character of the one-dimensional 
module $L(a)$, $a\in\C^\times$, are given by the formulas
$\chib\bigl(L(a)\bigr)=\chi_q\bigl(L(a)\bigr)=x_a$.

\medskip

Denote the Macmahon plane partition function by $\chi_0$:
\begin{align*}
\chi_0=\prod_{j=1}^\infty\frac{1}{(1-t^{-j})^j}.
\end{align*}
The character and the $q$-character of the Macmahon module are given by
\begin{align}
&\chib\bigl(\mathcal{M}(u,K)\bigr) =x_{K^{1/2}}\times \chi_0\,,\notag\\ 
&\chi_q\bigl(\mathcal{M}(u,K)\bigr) =x_{K^{1/2}}
\frac{X_{K^{-1}u}}{X_{u}}\sum_{\bs \lambda\in\bs\cP}
t^{-|\bs\lambda|} \prod_{\square\in\bs \lambda}A^{-1}_{q^{-\square} u}\,.
\label{M q-char}
\end{align}

These formulas follow from formula \eqref{Mac-psi} for the eigenvalues of 
$\psi^\pm(z)$ in the module $\mathcal{M}(u,K)$. 

\medskip

Similarly, the character and the $q$-character of the Fock module are given by
\begin{align}\label{F q-char}
&\chib\bigl(\F(u)\bigr) =x_{q}\times\frac{1}
{\prod_{j=1}^\infty(1-t^{-j})}\,,\\ 
&\chi_q\bigl(\F(u)\bigr) =x_{q}
\frac{X_{q_2^{-1}u}}{X_{u}}\sum_{\lambda\in\cP}
t^{-|\lambda|} \prod_{\square\in\lambda}A^{-1}_{q^{-\square} u}\,.
\end{align}

\medskip

To compute the character and the $q$-character of the
negative fundamental module, we first prove the following.
\begin{prop}\label{limit}
The negative fundamental module $M^-(u)$ is a limit of the Macmahon module modified by a one-dimensional module, 
\begin{align*}
M^-(u)=\lim_{K\to\infty} \left(L(K^{-1/2})\otimes_{\Delta^\perp}\cM(u,K)\right).
\end{align*}
\end{prop}
\begin{proof}
Extending the $\B^\perp$ action on $L(K^{-1/2})$,  
we define the action of the generators of $\E$ on it
by setting $e(z)=f(z)=0$ and $\psi^\pm(z)=K^{-1/2}$. 
Then all relations of $\E$ are satisfied with the exception of 
$\psi_0^+\psi_0^-=1$. 
We use this action on $L(K^{-1/2})$ and the coproduct $\Delta^\perp$ 
to define the action of all generators of $\E$ on 
 $M=L(K^{-1/2})\otimes_{\Delta^\perp}\cM(u,K)$. 

The underlying vector space of $M$
has a basis labeled by plane partitions.   
From the formulas \eqref{Mac-psi}, \eqref{Mac-f}, \eqref{Mac-e}, 
we see that the matrix coefficients of $f(z)$ acting in $M$
are independent of $K$, while 
those of $e(z)$ and $\psi^{+}(z)$ are polynomials in $K^{-1}$. 
Hence all generators have well defined limits as $K\to\infty$
and give a structure of a $\B^\perp$ module on $M$. 

We note that the limit of $\psi_0^-$ is zero, in particular 
it is not invertible, 
so we do not get the structure of an $\E$-module on $M$. 

Operator $\psi^+(z)$ have simple joint spectrum, and the
non-zero matrix coefficients of $e_r,f_{n}$ ($r>0,n\ge0$)
remain non-zero in the limit. 
It follows that $M$ is an irreducible $\B^\perp$ module 
with highest $\ell$-weight $(1-u/z)^{-1}$, and hence coincides with $M^-(u)$. 
\end{proof}

\begin{cor}\label{q char M-}
The character and the $q$-character of 
the negative  fundamental module are given by
\begin{align}
&\chib\bigl(M^-(u)\bigr) =\chi_0\,,\notag\\ 
&\chi_q\bigl(M^-(u)\bigr) =\frac{1}{X_{u}}\sum_{\bs \lambda\in\bs\cP}
t^{-|\bs\lambda|} 
\prod_{\square\in\bs \lambda}A^{-1}_{q^{-\square} u}\,.\label{M- q-char}
\end{align}
\end{cor}
\begin{proof}
By Proposition \ref{limit},  $\chi_q\bigl(M^-(u)\bigr)$ 
is obtained by taking the $K\to\infty$ limit of product of $x_{K^{-1/2}}$ and \eqref{M q-char}. 
\end{proof}
\medskip

\begin{prop}\label{M+ char}
The character of the positive  fundamental module is given by
\begin{align*}
&\chib\bigl({M}^+(u)\bigr) =\chi_0\,.
\end{align*}
\end{prop}
\begin{proof}
We have $M^+(u)=(L^{\vee}((1-u/z)^{-1}))^{*S^\perp}$
by \eqref{duality}, and it is clear that
\begin{align*}
\chib\bigl((L^{\vee}((1-u/z)^{-1}))^{*S^\perp}\bigr)
=\chib\bigl(L^{\vee}((1-u/z)^{-1})\bigr)\Bigl|_{t\to t^{-1}}
=\chib\bigl(L((1-u/z)^{-1}))\bigr)\,.
\end{align*}
Hence $\chib\bigl(M^+(u)\bigr)=\chib\bigl(M^-(u)\bigr)$.
\end{proof}

The $q$-character of positive fundamental module is very different 
from the $q$-character of negative fundamental module. 
We compute it later in 
Proposition \ref{prop:1-finite} (see also Corollary \ref{q char 1}). 
\medskip

\section{Finite type modules}\label{sec:finite-type} 
For a positive integer $k$, we say that a $\B^\perp$ module $V$
is {\it $k$-finite} if $\psi^+(z)$ has $k$ distinct eigenvalues on $V$. 
Tensor product of a $k_1$-finite module and a $k_2$-finite module
is at most $k_1k_2$-finite. 
We say $V$ is of finite type if it is $k$-finite for some $k$. 
Certainly finite dimensional modules are of finite type.  
But there exist also infinite dimensional modules of finite type. 

Denote by $\Fin$ the full subcategory of $\cO_{\B^\perp}$
consisting of all finite type modules.  
This category is the subject of this section.

\subsection{Modules with polynomial highest $\ell$-weight}
Consider highest $\ell$-weight $\B^\perp$
modules whose highest $\ell$-weights are polynomials in $z^{-1}$:
\begin{align}
M=L(\Psi),\quad \Psi(z)\in  \rg_{\B^\perp}\cap \C[z^{-1}]. 
\label{pol-hwt}
\end{align}
Abusing slightly the language, 
we say that $M$ has a polynomial highest $\ell$-weight.
Such modules have special properties which play a key role in 
the subsequent construction. 

Recall that the positive fundamental module $M^+(u)$ has a polynomial highest $\ell$-weight.

\medskip

The first property is polynomiality of currents acting on $M$.  
Introduce the notation for half currents
\begin{align*}
e_>(z)=\sum_{n=1}^\infty e_nz^{-n},\quad
f_{\ge}(z)=\sum_{n=0}^\infty f_n z^{-n}\quad \in \B^\perp[[z^{-1}]]\,.
\end{align*}

\begin{prop}\label{prop:poly-cur}
Let $M$ be as in \eqref{pol-hwt}. 
Then for each vector $w\in M$ we have
\begin{align}
e_{>}(z)w,\ f_{\ge}(z)w\in M\otimes \C[z^{-1}]\,,
\quad   \psi^+(z)w\in \Psi(z)\cdot M\otimes \C[z^{-1}].
\label{poly-cur}
\end{align}
\end{prop}

The second property is the existence of a tensor product
with respect to $\Delta$. 
\begin{prop}\label{prop:Delta}
Let $V\in\Ob\cO_\E$ be an $\E$ module, and let $M$ be as in  \eqref{pol-hwt}. 
Then the coproduct $\Delta$ gives a 
structure of $\B^\perp$ module on
the tensor product $V\otimes M$.
Denoting this module by $V\otimes_{\Delta} M$ 
we have 
$V\otimes_{\Delta} M\simeq V\otimes_{\Delta^\perp} M$. 
\end{prop}
The third property is concerned with the structure of submodules 
of $V\otimes_{\Delta} M$. 
\begin{prop}\label{prop:submod}
Let $V\in\Ob\cO_\E$ be an $\E$ module, and let $M$ be as in  \eqref{pol-hwt}. 
Assume that $V$ is irreducible with highest $\ell$-weight vector $v_0$. 
Then any proper submodule of $V\otimes_{\Delta}M$
has the form $V^{(0)}\otimes M$, where $V^{(0)}$ 
is an $\ell$-weighted linear subspace of $V$ which does not contain $v_0$.
\end{prop}

Proofs of these Propositions are technical, so we defer them to Appendix.  
Proposition \ref{prop:poly-cur} appears as Lemma \ref{lem:poly-cur}, 
Proposition \ref{prop:Delta}
as Lemma \ref{lem:Delta} and Corollary \ref{cor:Delta-Delta},
and  
Proposition \ref{prop:submod}
as Lemma \ref{lem:submod}, 
respectively.

\subsection{Grading on $M$}

Frenkel and Hernandez showed that positive fundamental modules 
for quantum affine algebras as vector spaces admit a grading  with favorable
properties (see \cite{FH}, Theorem 6.1).
This was a key step in their proof of polynomiality of $Q$ operators.
We show here that an analogous grading exists in the case 
of various modules of quantum toroidal $\gl_1$ algebra. 
In this section, our goal is Proposition \ref{prop:gradingM} below which
constructs the grading for modules with polynomial highest $\ell$-weights.

In this subsection, we set
\begin{align}
&M=L\bigl(\Psi\bigr)\,,\quad \Psi(z)=\prod_{j=1}^N(1-u_j/z)\,,
\label{MPsi}\\
&V=\F(u_1)\otimes_{\Delta^\perp}\cdots\otimes_{\Delta^\perp}\F(u_N)\,. 
\label{FFF}
\end{align}
We denote by $w_M$ the highest $\ell$-weight vector of $M$,
and by $\ket{\emptyset}_V$ the tensor product of $\ket{\emptyset}\in \F(u_i)$.
By Proposition \ref{prop:Delta}, 
the tensor product $V\otimes_\Delta M$ 
with respect to $\Delta$ is a well-defined $\B^\perp$ module.
We begin with some lemmas.

\begin{lem}\label{lem:SVM}
Let $V_{\lle -1}$ denote the subspace of $V$ in \eqref{FFF} consisting of vectors
of principal degree $\le -1$. 
Then the subspace $S=V_{\lle-1}\otimes M$ is a $\B^\perp$ submodule of 
$V\otimes_\Delta M$.
\end{lem}
\begin{proof}
Since $\B^\perp_-$ and $\B^\perp_0$ do not increase the principal degree, it is enough to 
show that $\B^\perp_+S\subset S$. The algebra $\B^\perp_+$ is generated by $e_n$ ($n\ge1$), $e^\perp_{-r}$ ($r\ge2$)
and $[e_0,e_2]$. We show $\Delta(x)S\subset S$ where $x$ is one of these elements. 

First consider the case of $e_n$ ($n\ge1$). We have for $v\otimes w\in S$
\begin{align*}
&\Delta(e_{>}(z))(v\otimes w)=(e(z)v\otimes\psi^+(z)w)_{>}+v\otimes e_{>}(z)w\,,
\end{align*}
where $\bigl(\sum_{n\in\Z}a_nz^{-n}\bigr)_{>}=\sum_{n>0}a_nz^{-n}$.
The second term  in the right hand side belongs to $S$. 
The first term also does if $\pdeg v<-1$. 
Suppose $\pdeg v=-1$. We may assume that $v=\ket{\emptyset}\otimes\cdots\otimes \ket{\square}\otimes\cdots\otimes \ket{\emptyset}$,
where $\ket{\square}\in\F(u_i)$ is a vector of degree $-1$. 
Using \eqref{Dperp-e1} in Appendix, 
we find 
\begin{align*}
e(z)v\in\C \ket{\emptyset}_V\times \delta(u_i/z)\prod_{j=i+1}^N\frac{q-q^{-1}u_j/z}{1-u_j/z}\,.
\end{align*}
On the other hand, 
Proposition \ref{prop:poly-cur} tells that $\psi^+(z)w$ is divisible by $\Psi(z)$.
Hence the first term vanishes. 

Next consider $\Delta(e^\perp_{-r})(v\otimes w)$ ($r\ge2$). 
This vector is a sum of terms of the form \eqref{XYXY} applied to $v\otimes w$. 
We are concerned only with terms whose first component is proportional to $\ket{\emptyset}_V$. 
Then the first component produces a delta function $\delta(u_i/z_{j_1'})$, 
while the second component contains $\psi^+(z_{j_1'})$. (Note that from the bilinear relation between 
$\psi^+(z)$ and $e(z)$ we have 
$(\ad e_0)^k \bigl(\psi^+(z)\bigr)\in \B^\perp[z^{-1}]\psi^+(z)\B^\perp[z^{-1}]$ for any $k\ge1$.)
Hence all such terms vanish. 

The case of $[e_0,e_2]$ is quite similar. 
\end{proof}

\begin{lem}\label{lem:gradation}
There exists a linear operator $\Phi\in \End M$ with the following properties. 
\begin{align}
&\text{$\Phi$ preserves the principal grading,}
\label{grad1}\\
&\Phi\circ x=q_2^{-\nu_2} x\circ \Phi
\quad (x\in \B^\perp_+\cap\E_{\nu_1,\nu_2})\,,
\label{grad2}\\
&\Phi \circ \overline{\psi}^+(z)=\overline{\psi}^+(q_2z)\circ \Phi
\quad \text{where $\overline{\psi}^+(z)=\Psi(z)^{-1}\psi^+(z)$.}
\label{grad3}
\end{align}
Moreover $\Phi$ is diagonalizable. Its eigenvalues 
have the form $q_2^{m}$ ($m\in\Z_{\ge0}$), and the eigenspace of 
$q_2^0$ is spanned by $w_M$.  
\end{lem}
\begin{proof}
We retain the notation of Lemma \ref{lem:SVM}. 
In the quotient $\bigl(V\otimes_\Delta  M\bigr)/S$,  
the image of $\ket{\emptyset}_V\otimes w_M$ generates a
submodule with irreducible quotient $L(\widetilde{\Psi})$, 
where $\widetilde{\Psi}(z)=q^N\Psi(q_2z)$. 
By \eqref{s-tau}, we have an isomorphism as a vector space 
\begin{align*}
\Phi_1=s_{q_2}^{-1}\circ\tau_{q^N}:L(\widetilde{\Psi}) \longrightarrow L(\Psi)=M\,.
\end{align*}
By comparing the characters, we find that there is an isomorphism
of $\B^\perp$ modules
\begin{align*}
\Phi_2:\bigl(V\otimes_\Delta  M\bigr)/S 
\overset{\sim}{\longrightarrow} L(\widetilde{\Psi})\,.
\end{align*}
Let further $\Phi_3$ denote the composition of the natural maps
\begin{align*}
M\longrightarrow \C\ket{\emptyset}_V\otimes M\hookrightarrow
V\otimes_\Delta M\longrightarrow \bigl(V\otimes_\Delta  M\bigr)/S\,.
\end{align*}
We set $\Phi=const. \Phi_1\circ\Phi_2\circ\Phi_3$,  
choosing $const.\in\C^\times$ so that  $\Phi w_M=w_M$.
By construction $\Phi$ preserves the principal grading, and we have 
\begin{align*}
&\Phi\circ e_r= q_2^{-r} e_r\circ \Phi\quad (r>0)\,,
\\
&\Phi\circ e^\perp_{-r}= q_2^{-1} e^\perp_{-r}\circ \Phi\quad (r>0)\,,
\\
&\Phi\circ [e_0,e_2]= q_2^{-2} [e_0,e_2]\circ \Phi\,,
\\
&\Phi\circ \psi^+(z)=\psi^+(q_2z)\circ\Phi\times 
\frac{\Psi(z)}{\Psi(q_2z)}\,.
\end{align*}
We have proved \eqref{grad1}, \eqref{grad2}, \eqref{grad3}.

In order to prove that $\Phi$ is diagonalizable, we use that the 
dual right module $M^{*}$ is generated from the lowest $\ell$-weight vector $w_M^*$
by $\B^\perp_+$.  We have $w_M^*\Phi=w_M^*$. 
Formula \eqref{grad2} means that any non-zero vector of the form
$w_M^*x_1\cdots x_k$, where $x_i\in \B^\perp_+$ are homogeneous elements, 
is an eigenvector of $\Phi$ with
eigenvalue $q_2^{\sum_{i=1}^k\hdeg x_i}$. Since they span 
 $M^{*}$, we conclude that $\Phi$ is diagonalizable on $M$, 
and that all eigenvalues have the form $q_2^m$, $m\in\Z_{\ge0}$. 
In particular, the eigenspace for $m=0$ is spanned by $w_M$. 
\end{proof}

\begin{prop}\label{prop:gradingM}
Let $M$ be as in \eqref{MPsi}. Then it admits a grading 
$M=\oplus_{m\ge0} M[m]$ as vector space,  with the following properties: for all $m\ge0$ we have 
\begin{align}
&x M[m]\subset M[m-\hdeg x] \quad (\forall x\in \B^\perp_+)\,,
\label{xM}\\
&\overline{\psi}^+_n M[m]\subset M[m-n]\quad (\forall n\ge0)\,,
\label{psiM}\\
&y M[m]\subset \sum_{j=0}^{-N\pdeg y}M[m-\hdeg y+ j] \quad 
(\forall y\in \B^\perp_-)\,,
\label{yM}\\
&M[m]=\oplus_{n\le0}M[m]\cap M_{n}\,,
\quad M[0]=\C w_M\,.
\label{M[m]-princ}
\end{align}
Here $x,y$ are assumed to be homogeneous elements, and 
we set $\Psi(z)^{-1}\psi^+(z)=\sum_{n\ge0}\overline{\psi}^+_n z^{-n}$. 
\end{prop}
\begin{proof}
Let $M[m]$ denote the eigenspace of $\Phi$ relative to the eigenvalue 
$q_2^m$. Then $M=\oplus_{m=0}^\infty M[m]$. 
The properties \eqref{xM}, \eqref{psiM} and \eqref{M[m]-princ} 
are immediate consequences of Lemma \ref{lem:gradation}. 
Note that \eqref{psiM} implies 
\begin{align}
\psi^+_n M[m]\subset \sum_{j=0}^N M[m-n+j].
\label{psiM2}
\end{align} 

Let us prove \eqref{yM}. 
For each $y\in\B^\perp_-$, 
statement \eqref{yM} is reduced to the following statement:
\begin{align}
&\text{For all $m\ge0$ and all homogeneous elements $x\in \B^\perp_+$ we have}\label{yM2}\\
&[x,y] M[m]\subset \sum_{j=0}^{-N\pdeg y}M[m-\hdeg (xy)+ j] \,.\nn
\end{align}
To see that \eqref{yM2} implies \eqref{yM}, 
take $v\in M[m]\cap M_n$ and write $yv=\sum_{l\ge0}w_l$, $w_l\in M[l]$. 
Without loss of generality we may assume that $\pdeg y<0$. 
Since $M$ is irreducible, for any $l_0$ such that $w_{l_0}\neq0$ we can take 
an $x\in \B^\perp_+$ satisfying $xv=0$ and $x w_{l_0}=w_M$. 
On the other hand, \eqref{yM2} implies 
\begin{align*}
\sum_{l\ge0}xw_l=xy v=[x,y]v\ \in \sum_{j=0}^{-N\pdeg y}M[m-\hdeg (xy)+ j].
\end{align*} 
Since $x w_{l_0}\in M[l_0-\hdeg x]$, we must have $m-\hdeg y\le l_0\le m-\hdeg y-N\pdeg y$. 
This implies \eqref{yM}. 

Note that
if \eqref{yM} holds for $y_1,y_2$ then it holds for $y_1y_2$. 
Hence it suffices to consider the case where $y$ is a generator of 
$\B^\perp_-$. 

Consider the case $y=f_n$ ($n\ge0$). 
If $x=e_r$ ($r\ge1$), then $[e_r,f_n]=\kappa_1^{-1}\psi^+_{r+n}$, hence
\eqref{psiM} applies. Next we take $x=e^\perp_{-r}$ ($r\ge2$) or $x=[e_0,e_2]$. 
If we set $H=\sum_{j\ge0}\C \psi^+_j$, then 
from the bilinear relation between $e$'s and $\psi^+$'s we obtain 
$\B^\perp_+H=H\B^\perp_+$
and $\ad e_0\bigl(\B^\perp_+H\bigr)\subset \B^\perp_+H$. 
With the aid of these relations we find
that $[e^\perp_{-r},f_n]\in \B^\perp_+H$ 
and $[[e_0,e_2],f_n]\in \B^\perp_+H$. 
Then \eqref{xM} and \eqref{psiM} apply, and therefore \eqref{yM2} holds for $y=f_n$. 
In particular, if \eqref{yM2} is true for $y$ then it is true also for $[f_0,y]$. 
Hence \eqref{yM2} holds for all $y$ in the subalgebra 
$\cN^\perp_{-}=\langle\bp_{-\nu_1,\nu_2}\, (\nu_1,\nu_2\ge1)\rangle$. 

It remains to show \eqref{yM2} for 
$y=\psi^\perp_k$ ($k\ge1$), or equivalently for 
$y=h^\perp_k$ ($k\ge1$) which are simpler to  work with. 
Let us consider the case $x=e^\perp_{-r}$ ($r\ge2$) by induciton on $r$. 
The commutator $[x,h^\perp_k]$ is proporional to $e^\perp_{-r+k}$. 
If $r\ge k$ then \eqref{xM} and \eqref{psiM} apply, 
and if $r<k$ then the induction hypothesis applies. 
To verify the cases $x=e_n$ ($n\ge2$) and $[e_0,e_2]$, it is sufficient to 
note the following. 
Suppose that the second line of \eqref{yM2} holds for
$x\in \B^\perp_+$ and $y=h^\perp_k$, then  
the same holds for $[h_1,x]$. 
This is because 
$[[h_1,x],h^\perp_k]\in \C [h_1,[x,h^\perp_k]]+\C [x,e^\perp_k]$
and $e^\perp_k\in\cN^\perp_{-}$. 

Proof of \eqref{yM} is now complete.
\end{proof}
\medskip

\noindent{\it Remark.}\quad Let $\br{h_n}_M$ denote the eigenvalue
of $h_n$ on the highest $\ell$-weight vector of $M$, and set
$\bar{h}_{n,M}=h_n-\br{h_n}_M$. Then \eqref{psiM} is equivalent to
\begin{align*}
\bar{h}_{n,M} M[m]\subset M[m-n]\quad
(\forall n\ge1)\,. 
\end{align*}
\qed

\subsection{$1$-finite modules}\label{subsec:positive}

We use Proposition \ref{prop:gradingM} 
to study the structure of the module $M$.

We start with positive fundamental modules.
\begin{prop}\label{prop:1-finite}
Module $M^+(u)$ is $1$-finite. 
We have
\begin{align}
\chi_q(M^+(u)) 
=X_u\times \chib(M^+(u))=X_u\times \chi_0\,.
\label{q-char M+}
\end{align}
\end{prop}
\begin{proof}
For each $n\le 0$, we take a basis of $M^+(u)_{n}$ by choosing a basis 
from each component $M[m]\cap M^+(u)_n$, $m\ge0$.  
In this basis, $\overline{\psi}^+(z)=\psi^+(z)/(1-u/z)$ is represented 
by a triangular matrix with $1$ on the diagonal
by Proposition \ref{prop:gradingM}.  
Therefore $\psi^+(z)$ 
has only one eigenvalue $1-u/z$, and 
$\chi_q\bigl(M^+(u)\bigr)$ has the stated form.
\end{proof}
\medskip

\noindent {\it Remark.}\quad 
It is instructive to think of  the result \eqref{q-char M+} 
as a formal limit of the $q$-character of the Macmahon module. 
Namely, rescaling $u$ to $Ku$ in \eqref{M q-char} and 
demanding the rule 
$\displaystyle{\lim_{K\to0}X_{Ku}=\lim_{K\to0}A_{Ku}=1}$, 
we have formally
\begin{align*}
\chi_q\bigl(M^+(u)\bigr)=
\lim_{K\to 0} x_{K^{1/2}}^{-1}\chi_q\bigl(\cM(Ku,K)\bigr)\,.
\end{align*}
\medskip

\begin{cor}\label{M-1}
Let $M=L(\Psi)$, $\Psi=\prod_{j=1}^N(1-u_j/z)$. Then $M$ is 1-finite. 
\qed
\end{cor}

The positive fundamental modules have the following description.
\begin{prop}\label{prop:presen-M}
Module $M^+(u)$ is the unique 
cyclic $\B^\perp$ module generated by $v_0$ satisfying 
\begin{align}
\B^\perp_+v_0=\C v_0\,,\quad \psi^+(z)v_0=(1-u/z)v_0,\quad  
\bp_{-n,l} v_0=0\quad (1\le n\le l).
\label{M-rel}
\end{align}
\end{prop}
\begin{proof}
Suppose $1\le n\le l$, and let $v=\bp_{-n,l} v_0$. 
By \eqref{yM} we have $v\in \sum_{j=0}^n M[-l+j]$. 
Since $M[m]=0$ for $m<0$, $M[0]=\C v_0$ and $n\ge 1$, we obtain $v=0$. 
Therefore the relations \eqref{M-rel} are satisfied in $M^+(u)$. 

Let $\widetilde{M}^+(u)$ denote the cyclic
$\B^\perp$ module defined by the relations \eqref{M-rel}.
Then there is a surjective morphism $\widetilde{M}^+(u)\to {M}^+(u)$.
Since $\widetilde{M}^+(u)$ is spanned by ordered monomials of 
$\bp_{-n,l}$ with $0\le l<n$ applied to $v_0$
(see Theorem \ref{thm:pbw}), we have 
\begin{align*}
\chib\bigl(\widetilde{M}^+(u)\bigr)\ll \chi_0\,.
\end{align*}
The right hand side equals 
$\chib\bigl({M}^+(u)\bigr)$ by Proposition \ref{M+ char}. Hence the two modules coincide. 
\end{proof}

Finally, we show that arbitrary tensor products of fundamental modules 
are irreducible. In the case of quantum affine algebras, this result is due to \cite{FH}. 
Our proof below is different from theirs and is equally applicable to that case.

 \begin{prop}\label{prop:irred-tensor}
For any $u_1,\cdots,u_N\in\C^\times$, the tensor product 
$M^+(u_1)\otimes_{\Delta^\perp}\cdots\otimes_{\Delta^\perp} M^+(u_N)$
is irreducible. 
Hence it is isomorphic to $L\bigl(\prod_{j=1}^N(1-u_j/z)\bigr)$.
\end{prop}
\begin{proof}
Set
$M=M^+(u_1)\otimes_{\Delta^\perp}\cdots\otimes_{\Delta^\perp} M^+(u_N)$.
By Corollary \ref{M-1}, $M$ is $1$-finite. 
The dual module $M^\vee=(M)^{*(S^\perp)^{-1}}$ is isomorphic to 
$M^{-\vee}(u_N)\otimes_{\Delta^\perp}\cdots\otimes_{\Delta^\perp} 
M^{-\vee}(u_1)$.  
By an analog of Corollary \ref{q char M-} for lowest 
$\ell$-weight modules, we have 
$\chi_q\bigl(M^\vee\bigr)=\mb(\Psi)^{-1}\bigl(1+\cdots\bigr)$,
where $\Psi(z)=\prod_{j=1}^N(1-u_j/z)$ 
and $\cdots$ stands for a sum of non-trivial 
monomials in the $A_{a}$'s. 

Suppose that $M^\vee$ has a non-trivial submodule $M^\vee_1$.   
From the structure of $\chi_q(M^\vee)$ mentioned above, 
we see that either $M^\vee_1$ or $M^\vee/M^\vee_1$ has a singular vector
of $\ell$-weight 
$\Phi(z)^{-1}$ different from $\Psi(z)^{-1}$. 
Let $M_1\subset M$ denote the orthogonal complement of $M^\vee_1$. 
From Lemma 3.1, we conclude that either $M/M_1$ or $M_1$ has
the $\ell$-weight $\Phi(z)$. 
Hence $M$ has two different 
$\ell$-weights $\Psi(z)$ and $\Phi(z)$. This is a contradiction.
Therefore $M^\vee$ is irreducible, and hence $M$ is irreducible. 
\end{proof}

\begin{cor}\label{q char 1} Let $M=L(\Psi)$, 
$\Psi=\prod_{j=1}^N(1-u_j/z)$. Then the $q$-character of $M$ is given by
\begin{align*}
\chi_q(M)=\chi_0^N \ \prod_{j=1}^N X_{u_j}.
\end{align*}
\qed
\end{cor}

\subsection{$2$-finite modules}

Now we proceed to discussing $2$-finite modules. 
Introduce the module 
\begin{align*}
N^+(u)=L\bigl(\Psi_{N^+(u)}\bigr)\,,
\quad \Psi_{N^+(u)}(z)=\frac{\prod_{i=1}^3(1-q_i^{-1}u/z)}{1-u/z}.
\end{align*}
Our goal in this subsection is to prove Theorem \ref{thm:N} below.

We make use of a construction similar to the one used in 
the proof of Lemma \ref{lem:SVM}. 
Let $\F(u)$ be the Fock module \eqref{Fock}
with the highest weight vector  $\ket{\emptyset}$.  
Let $\F(u)_{\lle-2}$ be the subspace of $\F(u)$
of principal degree $\le -2$. 
Set further $M(u)=M^+(q_3^{-1}u)\otimes_{\Delta^\perp} M^+(q_1^{-1}u)$. 
By Proposition \ref{prop:irred-tensor}, 
$M(u)$ is a $1$-finite module $L(\Psi)$
with highest $\ell$-weight $\Psi(z)=(1-q_3^{-1}u/z)(1-q_1^{-1}u/z)$. 
We consider the module 
$V=\F(u)\otimes_{\Delta} M(u) $ and its 
linear subspace $S=\F(u)_{\lle-2}\otimes M(u)$. 
In the following, $\ket{\square}$ stands for the vector $\ket{\lambda}\in\F(u)$ for the partition $\lambda=(1)$.

\begin{lem}\label{lem:N-con}
Notation being as above, $S$ is a $\B^\perp$ submodule of $V$. 
\end{lem}
\begin{proof}
Obviously $S$ is invariant under the action of $\B^\perp_-\B^\perp_0$. 
We show that $\B^\perp_+ S\subset S$. 
Let $v\in \F(u)_{\lle -2}$ and $w\in M(u)$. Then we have
\begin{align}
e_>(z)(v\otimes w)=(e(z)v\otimes \psi^+(z)w)_{>}+v\otimes e_{>}(z)w\,.
\label{e-v-w}
\end{align}
The second term in the right hand side belongs to $S$.
The first term also does unless $\pdeg v=-2$. 
If this is the case, then
$e(z)v\in (\C\delta(q_1^{-1}u/z)+\C\delta(q_3^{-1}u/z))\ket{\square}$. 
On the other hand, by Lemma \ref{lem:poly-cur} and 
the definition of $M(u)$,  
$\psi^+(q_3^{-1}u)=\psi^+(q_1^{-1}u)=0$ hold on $M(u)$. 
Therefore the first term vanishes and \eqref{e-v-w} belongs to $S$. 

By the same argument as in the proof of 
Lemma \ref{lem:SVM}, 
we see also that $e^\perp_{-r}(v\otimes w)\in S$ for $r>0$
and $[e_0,e_2](v\otimes w)\in S$. 
\end{proof}

\medskip
\begin{thm}\label{thm:N}
Module $N^+(u)$ is $2$-finite. 
Its $q$-character is given by 
\begin{align}
\chi_q(N^+(u))= \chi_0^2 \left(\frac{\prod_{i=1}^3 X_{q_i^{-1}u}}{X_u}
+t^{-1} \frac{\prod_{i=1}^3 X_{q_iu}}{X_{u}}\right)=\frac{\prod_{i=1}^3 X_{q_i^{-1}u}}{X_u}
\left(1+t^{-1} A_u^{-1}\right) \chi_0^2 \,.
\label{chiNM}
\end{align}
\end{thm}
\begin{proof}
Set $N=V/S$. As a linear space we have
$N=N_0\oplus N_1$, where 
$N_0=\C\ket{\emptyset}\otimes M(u)$,  
$N_1=\C\ket{\square}\otimes M(u)$ 
(we omit writing $\bmod\,  S$). 
Let $P_i:N\to N_i$ ($i=0,1$) be the projection. 

We have
\begin{align*}
&\psi^+(z)\bigl(\ket{\emptyset}\otimes w\bigr) 
=a_0(z)\ket{\emptyset}\otimes \psi^+(z) w\,,
\quad
\psi^+(z)\bigl(\ket{\square}\otimes w\bigr) 
=a_1(z)\ket{\square}\otimes \psi^+(z)w\,,
\end{align*}
where
\begin{align}
&a_0(z)=q\frac{1-q_2^{-1}u/z}{1-u/z}\,,
\quad 
a_1(z)=a_0(z)\prod_{i=1}^3\frac{1-q_iu/z}{1-q_i^{-1}u/z}\,.
\label{ai}
\end{align}
Hence $N_0,N_1$ are generalized eigenspaces of $\psi^+(z)$ with eigenvalues
\begin{align}
\Psi_0(z)=q\frac{\prod_{i=1}^3(1-q_i^{-1}u/z)}{1-u/z}\,, 
\quad
\Psi_1(z)=q\frac{\prod_{i=1}^3(1-q_iu/z)}{1-u/z}\,, 
\label{N-eigv}
\end{align}
respectively.

It is easy to see that for all $x=e_n$ ($n\ge1$),  
$e^\perp_{-r}$ ($r\ge2$) and $[e_0,e_2]$ we have
\begin{align}
&x\bigl(\ket{\emptyset}\otimes w\bigr) 
=\ket{\emptyset}\otimes x w\,,
\label{x-phi1}\\
&P_1 x\bigl(\ket{\square}\otimes w\bigr) 
=
\ket{\square}\otimes x w\,.
\label{x-phi2}
\end{align}
Similarly we compute
\begin{align*}
&f_{\gge}(z) \bigl(\ket{\square}\otimes w\bigr) 
=\bigl(\psi^-(z)\ket{\square}\otimes f(z)w\bigr)_{\gge} 
=\ket{\square}\otimes \bigl(a^{-}_1(z)f(z) w\bigr)_{\gge}\,,
\\
&P_0f_{\gge}(z) \bigl(\ket{\emptyset}\otimes w\bigr) 
=\bigl(\psi^-(z)\ket{\emptyset}\otimes f(z) w\bigr)_{\gge}
=\ket{\emptyset}\otimes \bigl(a^{-}_0(z)f(z) w\bigr)_{\gge}\,,
\end{align*}
where $a^{-}_i(z)$ denotes the expansion of the
rational function $a_i(z)$ at $z=0$. 
Let $\br{h_r}_{M}$ be the eigenvalue of $h_r$ on 
the highest weight vector $w_0$ of $M(u)$. 
Set $\bar{h}_{r,M}=h_r-\br{h_r}_{M}$ and 
\begin{align*}
&U_0
=\exp\bigl(\sum_{r=1}^\infty(1-q_2^r)\bar{h}_{r,M} u^{-r}\bigr)\,,
\quad 
U_1
=
U_0
\exp\bigl(-\sum_{r=1}^\infty \kappa_r\bar{h}_{r,M} u^{-r}\bigr)\,.
\end{align*}
Since $\bar{h}_{r,M}$'s are nilpotent, 
operators $U_i$
have a well defined action on $M(u)$, and  
we can write $a^{-}_i(z)f(z)=q^{-1} U_i f(z) U_i^{-1}$. 
Namely, for $y=f_n$ with $n\ge0$ we obtain
\begin{align}
&y \bigl(\ket{\square}\otimes w\bigr) 
=q^{-1}\ket{\square}\otimes U_1 yU_1^{-1} w\,,
\label{y-phi1}
\\
&P_0 y \bigl(\ket{\emptyset}\otimes w\bigr) 
=q^{-1}\ket{\emptyset}\otimes  U_0 yU_0^{-1} w\,. 
\label{y-phi2}
\end{align}
With a similar computation we have the same equations
for $y=\psi^{+,\perp}_r$ ($r>0$). 

Now we show that $N$ is irreducible. Let $W\subset N$ be a non-zero
$\B^\perp$ submodule. We have the decomposition into $\ell$-weight spaces
$W=W_0\oplus W_1$, $W_i=W\cap N_i$. 
Using \eqref{x-phi1}, \eqref{x-phi2}, and 
the irreducibility of $M(u)$ which follows from 
Proposition \ref{prop:irred-tensor}, we obtain either 
$\ket{\emptyset}\otimes w_0\in W_0$ 
or $\ket{\square}\otimes w_0\in W_1$.
It is easy to see that one implies the other, and hence both are satisfied. 
Since $U_i \B^\perp_-U_i^{-1} w_0=M(u)$, \eqref{y-phi1},\eqref{y-phi2}
imply that $W_i=N_i$, $i=0,1$. Hence $W=N$. 

Comparing the highest $\ell$-weight, we conclude that
$N=\tau_{q^{-1}}\bigl(N^+(u)\bigr)$. 
The assertion about the $q$-character is clear from 
\eqref{N-eigv} and the  proof above. 
\end{proof}

\medskip

\noindent{\it Remark.}\quad
Formula \eqref{chiNM} is equivalent to an identity in the Grothendieck ring $\Rep\, \B^\perp$,  
\begin{align}
[N^+(u)][M^+(u)]=\prod_{i=1}^3[M^+(q_i^{-1}u)]+
\Bigl(\prod_{i=1}^3[M^+(q_iu)]\Bigr)\{-1\}\,.
\label{NM-MMM}
\end{align}
Here $[V]$ means the class of $V$ in $\Rep\,\B^\perp$, and 
$V\{d\}$ stands for the module where $D^\perp$ acts as $q^dD^\perp$. 
The short exact sequence corresponding to \eqref{NM-MMM} reads
\begin{align*}
0\longrightarrow \bigl(\otimes_{i=1}^3 M^+(q_i u)\bigr)\{-1\}
\longrightarrow N^+(u)\otimes M^+(u)
\longrightarrow \otimes_{i=1}^3 M^+(q_i^{-1}u)
\longrightarrow 0\,
\end{align*}
where $\otimes=\otimes_{\Delta^\perp}$. 
\qed
\medskip

We show that the module $N^+(u)$ also possesses 
the grading similar to the one established in 
Proposition \ref{prop:gradingM} for modules with 
polynomial highest $\ell$-weight.  

\begin{prop}\label{prop:gradingN}
Let $N^+(u)=N_0\oplus N_1$ be the decomposition into $\ell$-weight
spaces, where $N_0$ corresponds to the highest $\ell$-weight. 
Then there exist gradings 
$N_0=\oplus_{m=0}^\infty N_0[m]$, 
$N_1=\oplus_{m=1}^\infty N_1[m]$
as vector space, 
with the following properties:
\begin{align}
&x N_0[m]\subset N_0[m-\hdeg x]\quad (x\in \B^\perp_+)\,,
\label{xN0}
\\
&x N_1[m]\subset N_1[m-\hdeg x]+\sum_{j=0}^{m}N_0[j]\quad (x\in \B^\perp_+)\,,
\label{xN1}\\
&\overline{\psi}^+_{r,i} N_i[m]\subset N_i[m-r]\quad (r>0,\ i=0,1)\,,
\label{psiN}\\
&y N_0[m]\subset \sum_{j=0}^{m-\hdeg y-2\pdeg y}N_0[j]
+ \sum_{j=1}^{m-2\pdeg y-1}N_1[j]\quad (y\in \B^\perp_-)\,,
\label{yN0}\\
&y N_1[m]\subset  \sum_{j=1}^{m-\hdeg y-2\pdeg y}N_1[j]
\quad (y\in \B^\perp_-)\,.
\label{yN1}
\end{align}
Here $x,y$ are assumed to be homogeneous, and 
\begin{align*}
\Psi_i(z)^{-1}\psi^+(z)=\sum_{r\ge0}\overline{\psi}^{+}_{r,i}z^{-r}\,,
\end{align*}
where $\Psi_i(z)$ are given by \eqref{N-eigv}.
\end{prop}
\begin{proof}
As in the proof of Theorem \ref{thm:N}, 
We use the realization of $N^+(u)$ as a quotient of 
$\F(u)\otimes_{\Delta} M(u)$, 
$M(u)=M^+(q_3^{-1}u)\otimes_{\Delta^\perp}M^+(q_1^{-1}u)$.
We have $N_0=\ket{\emptyset}\otimes M(u)$
and 
$N_1=\ket{\square}\otimes M(u)$. 
Using the grading $M(u)=\oplus_{m=0}^\infty M[m]$ 
in Proposition \ref{prop:gradingM}, 
we define 
\begin{align*}
N_0[m]= \ket{\emptyset}\otimes M[m]\,,
\quad 
N_1[m]= \ket{\square}\otimes M[m-1]\,.
\end{align*}
Then \eqref{psiN} holds by definition. 
From \eqref{xM}, \eqref{x-phi1} and \eqref{x-phi2} we obtain 
$P_i x N_i[m]\subset N_i[m-\hdeg x]$ for $i=0,1$ and 
$x=e_n,e^\perp_{-r},[e_0,e_2]$. 
Take $w\in M[m-1]$. Using $e(z)\ket{\square}=\alpha\delta(u/z)\ket{\emptyset}$
$(\alpha\in\C^\times)$ we compute
\begin{align*}
P_0\, \Delta(e_{>}(z))\bigl(\ket{\square}\otimes w\bigr)
=\alpha \frac{u/z}{1-u/z}\ket{\emptyset}\otimes \psi^+(u)w\,.
\end{align*}
Using \eqref{psiM2}, we get $P_0e_nN_1[m]\subset \sum_{j=0}^m N_0[j]$. 
We have also
\begin{align*}
P_0\, \Delta\bigl((\ad e_0)^{r-1}e_1\bigr)\bigl(\ket{\square}\otimes w\bigr)
&\in \sum_{j=1}^{r-1}\sum_{p\ge1}\C \ket{\emptyset}\otimes
(\ad e_0)^{j-1}\ad \psi^+_p(\ad e_0)^{r-1-j}e_1 \cdot w
\\
&+\C
\ket{\emptyset}\otimes \sum_{p\ge1} (\ad e_0)^{r-1}\psi^+_p \cdot w\,,
\end{align*}
which implies $P_0 e^\perp_{-r}N_1[m]\subset\sum_{j=0}^{m-1}N_0[j]$. 
Arguing similarly for $[e_0,e_2]$, 
we find that \eqref{xN0}, \eqref{xN1} are satisfied for  $x\in \B^\perp_+$. 
The general case $x\in \B^\perp_+$ follows from this by applying
$\ad h_1$ and using \eqref{psiN}. 

It remains to show \eqref{yN0} and \eqref{yN1}. 
From \eqref{yM}, 
\eqref{y-phi1} and \eqref{y-phi2}, we obtain
\begin{align*}
P_i y N_i[m]\subset \sum_{j=0}^{m-\hdeg y-2\pdeg y}N_i[j]
\quad (i=0,1,\ y=f_n,\psi^{+,\perp}_r).
\end{align*}
Also, by a computation similar to the one for $P_0e_nN_1[m]$ given above,  
we can treat the case $P_1 f_n N_0[m]$.
Let us consider $P_1\psi^{+,\perp}_rN_0[m]$. 
For $w\in M[m]$ we are to compute 
\begin{align*}
P_1[\Delta f(z_1),\cdots [\Delta f(z_{r-1}),\Delta f(z_r)]\cdots] 
\bigl(\ket{\emptyset}\otimes w\bigr)
\end{align*}
and take the coefficient of $z_1^1z_2^0\cdots z_{r-1}^0z_r^{-1}$. 
Writing $\ad X=L(X)-R(X)$, where $L(X)$ (resp. $R(X)$)
signifies the left (resp. right) multiplication, 
we obtain sums of terms
\begin{align*}
&\bigl(\prod_{k=j+1}^ra_0^-(z_k)-\prod_{k=j+1}^ra_1^-(z_k)\bigr) 
\ket{\square}\otimes
A_1\cdots A_{j-1} \ad f(z_{j+1})\cdots \ad f(z_{r-1})
f(z_{r})\cdot w\quad (1\le j\le r-1),
\end{align*}
and $\ket{\square}\otimes A_1\cdots A_{r-1}1\cdot w$ 
for $j=r$. 
Here $a_i^-(z)$ are the expansions of $a_i(z)$ in \eqref{ai} at $z=0$, and 
$A_k=a^-_1(z_k)L(f(z_k))-a^-_0(z_k)R(f(z_k))$. 
From these we can 
check that $P_1\psi^{+,\perp}_rN_0[m]\subset \sum_{j=1}^{m+2r-1}N_1[j]$. 
\end{proof}

\subsection{More on $q$-characters}\label{q char 2}
We discuss the combinatorics of $q$-characters.

The following lemma, which explains the role of $A_a$, 
is a direct analog of Lemma 3.1 in \cite{Y}.
\begin{lem}\label{Y}
Let $V\in\Ob \cO_{\E}$ and let $\Psi,\Phi$ be $\ell$-weights of $V$. 
Assume that $\bigl(f(z) V_{(\Psi,n)}\bigr)\cap V_{(\Phi,m)}\neq 0$. 
Then $m=n-1$, and 
there exists an $a\in\C^\times$ such that $\bm(\Phi)=\bm(\Psi)A^{-1}_a$. 
Moreover, there exist bases $\{v_k\}$ and $\{w_l\}$ of the
$\ell$-weight spaces $V_{(\Psi,n)}$ and $V_{(\Phi,m)}$, respectively, 
such that 
\begin{align*}
f(z)v_k=\sum_l P_{k,l}\bigl(\partial_a\bigr)\delta(a/z) w_l+\cdots\,.
\end{align*}
Here $P_{k,l}(\partial_a)$ is a polynomial in $\partial_a=\partial/\partial a$ 
of degree at most $k+l-2$, and $\cdots$ stands for a sum of 
terms which belong to $\ell$-weight spaces other than $V_{(\Phi,m)}$. 
\end{lem}
A similar lemma holds for the action of $e(z)$ 
with the replacement of $A^{-1}_a$ by $A_a$ and of equation $m=n-1$ by $m=n+1$. 
\medskip

In the following, for $\Psi(z)=p(z)/q(z)\in \rg_{\B^\perp}$ with $p(z),q(z)\in\C[z^{-1}]$,  
we set $d(\Psi)=\deg_{z^{-1}} p-\deg_{z^{-1}} q$. 
\begin{lem}\label{lem:q-char V}
Let $\Psi\in\rg_{\B^\perp}$ and $V=L(\Psi)$. 
Then the $q$-character of $V$ has the form
\begin{align}\label{1+A}
\chi_q(V)=\bm(\Psi)(1+\sum_i {\bs m_i})\,\chi_0^{\max(d(\Psi),0)},
\end{align} 
where 
each monomial $\bs m_i$ is a product of $t^{-1}A^{-1}_a$'s.
\end{lem}
\begin{proof}
If $d(\Psi)\le 0$, then $V$ is a subquotient of a
tensor product of several Macmahon modules $\cM(u,K)$ and negative 
fundamental modules $M^-(u)$. 
The $q$-characters of $\cM(u,K)$ and $M^-(u)$ 
have the required form, see \eqref{M q-char} and \eqref{M- q-char}. 
Hence the lemma is true in this case. 

If $d(\Psi)>0$, then $V$ is a subquotient of a
tensor product of several Macmahon modules with a 
module which has a polynomial highest $\ell$-weight.
Then the lemma follows from Proposition \ref{prop:submod}.
\end{proof}
We have the classification of $1$-finite modules.
\begin{prop}\label{1}
Let $\Psi\in\rg_{\B^\perp}$ and $V=L(\Psi)$. Then $V$ is 
$1$-finite if and only if $V$ has polynomial highest $\ell$-weight. 
\end{prop}
\begin{proof}
It suffices to prove the `only if' part.

Suppose $d(\Psi)\le 0$. From \eqref{1+A}, $V$ is $1$-finite if and only if 
it is one-dimensional. It is easy to see that one-dimensional modules are exactly $L(a)$, 
$a\in\C^\times$. Hence the assertion is true in this case. 

Let $d(\Psi)>0$. We can write $\Psi(z)=\Psi_1(z)\Psi_2(z)$
where $\Psi_1(z)\in \rg_\E$, $\Psi_2(z)\in\rg_{\B^\perp}\cap\C[z^{-1}]$.  
Set  $W=L(\Psi_1)$ and $M=L(\Psi_2)$.  
In view of Proposition \ref{prop:submod}, $V$ is a quotient of the module $W\otimes_{\Delta}M$ by a   
submodule of the form $W^{(0)}\otimes_{\Delta}M$, where $W^{(0)}$ does not contain the highest $\ell$-weight
vector of $W$. 
We may assume that the poles of $\Psi_1(z)$ do not overlap with the zeroes of $\Psi_2(z)$. 
Then by Lemma \ref{Y} we must have that $W^{(0)}\cap W_{-1}=0$. 
Therefore, if $V$ is $1$-finite, then $\Psi_1(z)=1$. 
\end{proof} 

\subsection{Conjectures on finite type modules}\label{conj}  

Introduce the module 
\begin{align*}
N_{i,j,k}^+(u)=L\bigl(\Psi_{N_{i,j,k}^+(u)}\bigr)\,,
\quad 
\Psi_{N^+_{i,j,k}(u)}(z)
=\frac{(1-q_3^{-i}u/z)(1-q_1^{-j}u/z)(1-q_2^{-k}u/z)}{1-u/z}.
\end{align*}
We have $N^+(u)=N^+_{1,1,1}$.

\begin{lem} For any $i,j,k\in\Z_{>0}$, the module $N_{i,j,k}^+(u)$ is of finite type.
\end{lem}
\begin{proof}
Let $F(u)=L\bigl((1-q_2^{-k}u/z)/(1-u/z)\bigr)$. 
By \cite{FFJMM2}, this module has a basis of plane partitions 
with at most $k$ layers:
$\bs \la=(\la^{(1)},\dots,\la^{(k)},\emptyset,\emptyset,\cdots)$. 
Let $M=M^+(uq_3^i)\otimes_{\Delta^\perp} M^+(uq_1^j)$.
 We consider the tensor product $V=F(u)\otimes_{\Delta} M$. 
Let $S\subset V$ be the subspace spanned by vectors of the form 
$\ket{\bs \la} \otimes v$, where $v\in M$ and 
$\bs \la$ is such that either $\la^{(1)}_1>j$ or 
$\la^{(1)}_{i+1}>0$.
Then, arguing as in Lemma \ref{lem:N-con}, we see that 
$S\subset V$ is a submodule. Clearly $V/S$ has finite type. 
\end{proof}
 
\begin{conj}
The completed Grothendieck ring of $\Fin$ is topologically generated 
by $[N^+_{i,j,k}(u)]$ with $i,j,k\in\Z_{>0}$, $u\in\C^\times$, 
and $[M^+(u)]$, $u\in\C^\times$.
\end{conj}
\medskip

Next, we discuss the grading. Recall that 
we constructed the grading for modules with polynomial highest $\ell$-weights, 
see Proposition \ref{prop:gradingM}, and for $N^+(u)$, 
see Proposition \ref{prop:gradingN}. This grading has the following property. 
The annihilation operators 
and the modified current $\overline{\psi}^+(z)$ respect the grading given by their homogeneous degrees. 
The action of creation operators is not graded, 
but it changes the grading in a controllable way.
We expect that such a grading exists for all finite type modules.

An interesting question is to compute the formal character of 
a finite type module with respect to this grading.
\begin{conj} For positive fundamental module $M=M^+(u)$ we have
\begin{align*}
\sum_{n,m} (\dim M_n[m]) t^nr^m
=\prod_{j=1}^\infty\prod_{i=0}^{j-1}\frac{1}{1-t^{-j}r^{i}}.
\end{align*}
\end{conj}
\medskip

Finally, we suggest a cluster algebra structure on the category $\Fin$ in the spirit of \cite{HL}.
For generalities on cluster algebras, we refer to \cite{FZ}. 

Recall that a quiver $\Ga$ is an oriented graph. 
It is given by a set of vertices $\Ga_0$, a set of arrows $\Ga_1$ and maps s,t: 
$\Ga_1\to \Ga_0$, called source and target, respectively. 
We will assume that $\Ga$ has no loops and no $2$-cycles. 
That is for any $\alpha\in\Ga_1$ we have $s(\alpha)\neq t(\alpha)$, 
and for any $\alpha_1,\alpha_2\in \Ga_0$, 
we have $t(\alpha_1)=s(\alpha_2)$ 
implies $s(\alpha_1)\neq t(\alpha_2)$.

Given a domain $R$, 
a cluster is a pair $(\Ga,c)$, where $\Ga$ is a quiver and $c$ is a map $c: \Ga_0\to {R}$. The image of $c$ is called the set of cluster variables.

Let $R=\Rep\, \Fin$ be the Grothendieck ring of the 
category of finite type modules. 
A mutation of the cluster $(\Ga,c)$ in the direction 
$\ga\in\Ga_0$ is a new cluster with the following properties. 
The new quiver has the same set of vertices $\Ga_0$. The set of new arrows is obtained from $\Ga_1$ by the following three steps:
\begin{enumerate}
\item for each subquiver $\ga_1\to\ga\to\ga_2$ add a new arrow $\ga_1\to\ga_2$;
\item reverse all arrows with source or target $\ga$;
\item 
remove all 2-cycles.
\end{enumerate}
The new map $c_\gamma:\Ga_0\to R$ is given by $c_\ga(\ga_1)=c(\ga_1)$ if $\ga_1\neq \ga$ and 
\begin{align}\label{mut}
c_\ga(\ga)c(\ga)=\prod_{\al\in\Ga_1,\ s(\al)=\ga} c(t(\al))+\left(\prod_{\al\in\Ga_1,\ t(\al)=\ga} c(t(\al))\right)\{d\}
\end{align}
for some $d\in\Z_{\geq 0}$.
If such $d$ and the new cluster variable 
$c_\ga(\ga)\in R$ exist, they are unique, and 
we say that the cluster $(\Ga,c)$ 
{\it can be mutated in the direction $\ga\in\Ga_0$}.

\medskip

We define the quiver as follows: 
$\Ga_0=\Z^3$ and $(i_1,j_1,k_1)\to(i_2,j_2,k_2)$ 
if and only if 
$(i_2,j_2,k_2)\in\{(i_1+1,j_1,k_1), \ (i_1,j_1+1,k_1),\ (i_1,j_1,k_1+1)\}$. 
For $\Box=(i,j,k)\in\Z^3$ 
we set $q^\Box=q_3^iq_1^jq_2^k$ as before.

Then for $a\in\C^\times$, we define the seed --- 
or the initial cluster --- $(\Ga,c_0)$ 
by sending vertices of our graph to the positive fundamental modules:
\begin{align}\label{seed}
c_0:\ \Z^3\to  \Rep\, \Fin, \qquad \Box \mapsto [M^+(aq^\Box)].
\end{align}

\begin{lem}
The cluster $(\Ga,c_0)$ can be mutated in any direction $\Box\in\Ga_0$. 
The new cluster variable is the class of the $2$-finite
module $[N^+(aq^\Box)]$.
\end{lem}
\begin{proof}
The mutation equation \eqref{mut} in the case of the lemma coincides with \eqref{NM-MMM}.
\end{proof}
We also note that by Proposition \ref{prop:irred-tensor} a product of any modules in the initial cluster is irreducible.
We have the following conjecture motivated by \cite{HL}.

A $\B^{\perp}$ module $V$ is called {\it prime} if it cannot be written as tensor product 
of two modules both of which are not one-dimensional. 
A $\B^{\perp}$ module $V$ is called {\it real} if $V\otimes_{\Delta^\perp} V$ is irreducible. 
A $\B^{\perp}$ module $V$ is called {\it normalized} if $V_0\neq 0$ and $V_n=0$ for $n\in\Z_{>0}$.

\begin{conj} (i) The cluster $(\Ga,c_0)$ can be repeatedly mutated in any sequence of directions. Each mutation corresponds to a short exact sequence in category $\Fin$.

(ii) Every cluster variable obtained through a sequence of mutations is an 
irreducible prime real module.

(iii) Every irreducible prime real normalized module appears as a cluster variable.

(iv) Any tensor product module corresponding to a 
cluster variable in the same cluster is irreducible. 
In particular, both terms in the left hand side of mutation equation
\eqref{mut} correspond to irreducible modules. 
\end{conj}

\noindent

{\it Remark.}\quad
For quantum affine algebras,  
Hernandez and Leclerc found a cluster algebra structure
of the Grothendicek ring of $\cO^+$, which is the category 
generated by finite dimensional modules and the positive fundamental modules, 
see the announcement in \cite{H}.
It can be shown that $\cO^+$ coincides with the category of finite type modules for quantum affine algebras.
Also there exist $2$-finite modules analogous to 
$N^+(u)$ which satisfy $2$-term TQ relations, allowing for intepretation as mutations. 
We hope to discuss these issues elsewhere.

\section{Bethe ansatz}\label{sec:Bethe}

\subsection{Universal $R$ matrix}
 
As was mentioned in Subsection \ref{subsec:Borel}, 
algebra $\E$ is a quotient of the quantum double of its Borel subalgebra 
$\B^\perp$, and hence is equipped with the universal $R$ matrix. 
For a technical reason (see Remark after Proposition \ref{prop:poly-M} below), 
we use the $R$ matrix $\cR$ associated with the opposite coproduct 
$\Delta^{\perp,\mathrm{op}}=\sigma\circ\Delta^\perp$, where $\sigma(a\otimes b)=b\otimes a$.
Its main properties are as follows. 
\begin{align}
&\cR\ \Delta^{\perp,\mathrm{op}}(x)=\Delta^{\perp}(x)\ \cR\quad (x\in \E)\,,
\label{int-R}\\
&\bigl(\Delta^{\perp,\mathrm{op}}\otimes\id\bigr)\cR=\cR_{1,3}\,\cR_{2,3}\,,
\quad
\bigl(\id\otimes\Delta^{\perp,\mathrm{op}}\bigr)\cR=\cR_{1,3}\,\cR_{1,2}\,,
\label{copro-R}\\
&\cR_{1,2}\cR_{1,3}\cR_{2,3}=\cR_{2,3}\cR_{1,3}\cR_{1,2}\,,
\label{YBE}
\end{align}
where, as usual, the suffixes $i,j$ of $\cR_{i,j}$ 
stand for the tensor components, e.g., $\cR_{1,2}=\cR\otimes 1$. 

Element $\cR$ has the form (see \cite{BS}, \cite{Ng})
\begin{align*}
&\cR=q^{t_\infty}\cR_+\cR_0\cR_-\,.
\end{align*}
The factors $\cR_{\pm}$ can be written as 
\begin{align}
&\cR_+=1+\sum_{\nu_1>0\atop \nu_2>0} \sum_{i=1}^{N^+_{\nu_1,\nu_2}}
x^{(i)}_{\nu_1,\nu_2}\otimes y^{(i)}_{-\nu_1,-\nu_2} 
\quad \in \B^\perp_+\widehat{\otimes}\Bb^\perp_-\,,
\label{cR+}\\
&\cR_-=1+\sum_{\nu_1>0\atop \nu_2\ge0}  
\sum_{i=1}^{N^-_{\nu_1,\nu_2}}
x^{(i)'}_{-\nu_1,\nu_2}\otimes y^{(i)'}_{\nu_1,-\nu_2}
\quad \in \B^\perp_-\widehat{\otimes}\Bb^\perp_+\,,
\label{cR-}
\end{align}
where the suffixes $\nu_1,\nu_2$ indicate bidegrees. 
The middle factor $\cR_0$ is given by
\begin{align}
 &\cR_0=\exp\Bigl(-\sum_{r=1}^\infty
r \kappa_r h_r\otimes h_{-r}\Bigr)\,.
\label{R0}
\end{align}
Finally $q^{t_\infty}$ 
is defined formally by 
\begin{align}
q^{t_\infty}=q^{c^\perp\otimes d^\perp+d^\perp\otimes c^\perp}\,,
\label{tinf}
\end{align}
where $C^\perp=q^{c^\perp}$, $D^\perp=q^{d^\perp}$. 
In what follows we consider tensor product modules 
$V\otimes_{\Delta^\perp}W$ with $V\in \Ob\cO_{\B^\perp}$ and 
$W\in \Ob\cO_\E$. 
Formula \eqref{tinf} gives a well defined operator on such modules. 

\subsection{Normalized $R$ matrix and polynomiality}

In this section, we take $u$ to be an indeterminate. 

We define $s_u:\E\to\E[u^{\pm1}]$ by setting 
$s_u(x)=u^{\hdeg x}x$ for any homogeneous element  $x\in \E$.
We set also 
\begin{align*}
 \cR(u)=\bigl(s_u\otimes\id\bigr)\bigl(\cR\bigr)\quad
\in \B^\perp\widehat{\otimes}\Bb^\perp[[u]]\,.
\end{align*}
In formula \eqref{cR+} for $\cR_+$,  
the first tensor component of each term acts as
annihilation operator on modules from $\cO_{\B^\perp}$.
Likewise, in formula \eqref{cR-} for $\cR_-$,  
the second tensor component of each term acts as
annihilation operator on modules from $\cO_{\E}$. 
Therefore, if $V\in \Ob\cO_{\B^\perp}$ and $W\in \Ob\cO_\E$, then
each coefficient of the formal series $\cR(u)$ 
is a well defined operator on $V\otimes_{\Delta^\perp} W$.  

Suppose further that $V$ is a tensor product of highest $\ell$-weight $\B^\perp$ modules,   
and  $W$ is a tensor product of highest $\ell$-weight $\E$ modules.  
Denote by $v_0\in V$ the tensor product of highest $\ell$-weight vectors, and 
by $w_0\in W$ the tensor product of highest $\ell$-weight vectors. 
We write the eigenvalues of $h_r$ on these vectors as $\br{h_r}_V$, $\br{h_r}_W$, respectively. 
Suppose that $C^\perp$ acts on $V$ as a scalar $C_V^\perp$ and on $W$ as a scalar $C_W^\perp$.
From the remark above, we see that 
$\cR(u)\bigl(v_0\otimes w_0\bigr)=f_{V,W}(u)\bigl(v_0\otimes w_0\bigr)$, 
where
\begin{align}
f_{V,W}(u)=(C^\perp_V)^{\pdeg w_0}(C^\perp_W)^{\pdeg v_0}
\exp\bigl(-\sum_{r>0}u^r r\kappa_r \br{h_r}_V\br{h_{-r}}_W\bigr)\,.
\label{fVW}
\end{align}
We have
\begin{align*}
f_{V_1\otimes V_2,W}(u)=f_{V_1,W}(u)f_{V_2,W}(u)\,,
\quad
f_{V, W_1\otimes W_2}(u)=f_{V,W_1}(u)f_{V,W_2}(u)\,.
\end{align*}
For example, 
\begin{align*}
&f_{M^+(a),\F(b)}(u)=\exp\Bigl(\sum_{r>0}\frac{1-q_2^r}{r\kappa_r}\Bigl(\frac{ua}{b}\Bigr)^r\Bigr)
\,,\\
&f_{\F(a),\F(b)}(u)=\exp\Bigl(-\sum_{r>0}\frac{(1-q_2^r)(1-q_2^{-r})}{r\kappa_r}\Bigl(\frac{ua}{b}\Bigr)^r\Bigr)
\,.
\end{align*}

We define the normalized $R$ matrix 
$\Rb_{V,W}(u)\in \End\bigl(V\otimes W\bigr)[[u]]$ by 
\begin{align*}
\Rb_{V,W}(u)\bigl(v\otimes w\bigr)=
f_{V,W}(u)^{-1}\cR(u)\bigl(v\otimes w\bigr)\quad
(v\in V, w\in W).
\end{align*}
In the next two Propositions we study the polynomial nature of the
$\Rb_{V,W}(u)$ and its growth order at $u\to\infty$. 
This will be used in the next subsection
to discuss the polynomiality of the transfer matrix.

For vectors $w_1\in W^*$ and $w_2\in W$, introduce the 
notation $L_{w_1,w_2}(u)$ for the matrix coefficients 
in the second component, 
\begin{align*}
v_1L_{w_1,w_2}(u)v_2
=v_1\otimes w_1\Rb_{V,W}(u)
v_2\otimes w_2
\quad (v_1\in V^*, v_2\in V).
\end{align*} 
We regard $V^*$, $W^*$ as right $\B^\perp$ modules. 
The intertwining property \eqref{int-R} of the $R$ matrix implies that
\begin{align}
L_{w_1, h^\perp_rw_2}(u)
&=(C_V^\perp)^{-r}L_{w_1h^\perp_r,w_2}(u)
+[h^\perp_r,L_{w_1,w_2}(u)]_{(C_W^\perp)^{-r}}\,, 
\label{Lhr}\\
L_{w_1e^\perp_n,w_2}(u)
&=u [L_{w_1,w_2}(u),e^\perp_n]_{(C_W^\perp)^n}+
L_{w_1,e^\perp_{n}w_2}(u)(C_V^\perp)^n
\label{Len}\\
&+\sum_{j\ge 1}L_{w_1,e^\perp_{n-j}w_2}(u)\psi^{+,\perp}_j
(C_V^\perp)^n
-(C_W^\perp)^n u\sum_{j\ge 1}e^\perp_{n-j}L_{w_1\psi^{+,\perp}_{j},w_2}(u)\,,
\nn
\end{align}
for all $r>0$ and $n\in\Z$. 
Here we write $[A,B]_p=AB-pBA$.

We shall say that a linear operator on $V$ is polynomial if 
it acts as a polynomial in $u$ on each subspace $V_n$,  
with degree possibly depending on $n$. 

\begin{prop}\label{prop:poly-M}
Notation being as above, consider the special case
$V=M^+(1)$, $W=\F(v)$, $C^\perp_V=1$, $C_W^\perp=q^{-1}$. Let  
$V=\oplus_{m=0}^\infty M[m]$ be the grading in 
Proposition \ref{prop:gradingM}, and denote by $u^\partial$ the operator 
$u^\partial\bigl|_{M[m]}=u^m\times\id_{M[m]}$. 
Then $L_{w_1,w_2}(u)$ is polynomial 
for any $w_1\in \F(v)^*$, $w_2\in \F(v)$.   
Moreover we have 
\begin{align}
u^\partial L_{w_1,w_2}(u)u^{-\partial} M[m]
\subset u^s\cdot \sum_{j=0}^{m+s}M[j]\otimes\C[u^{-1}]
\quad (s=-\pdeg w_2).
\label{growthM}
\end{align}
\end{prop}
\begin{proof}
Consider first the case $w_1=\bra{\emptyset}$ and $w_2=\ket{\emptyset}$. 
We have $r\kappa_r\br{h_{-r}}_W=(1-q_2^r)v^{-r}$.  
From \eqref{R0}, \eqref{fVW} and \eqref{tinf} we obtain
\begin{align*}
L_{\bra{\emptyset},\ket{\emptyset}}(u)=
q^{-d^\perp}\exp\Bigl(-\sum_{r>0}\Bigl(\frac{u}{v}\Bigr)^{r}
(1-q_2^r)\bar{h}_{r,V}\Bigr)\,,
\end{align*}
where $\bar{h}_{r,V}=h_r-\br{h_r}_V$. 
Thanks to \eqref{psiM}, on each degree subspace, the operator  
$\sum_{r>0}(u/v)^r(1-q_2^r)\bar{h}_{r,V}$
is a finite sum and is nilpotent. 
Hence $L_{\bra{\emptyset},\ket{\emptyset}}(u)$ is polynomial. 
Since
\begin{align}
u^\partial \bar{h}_{r,V} u^{-\partial} M[m]\subset u^{-r} M[m-r]\,,
\label{hbarV}
\end{align}
operator 
$u^\partial L_{\bra{\emptyset},\ket{\emptyset}}(u)u^{-\partial}$
is independent of $u$, and 
\eqref{growthM} holds true for this element. 

Since $W=\F(v)$ is generated from $\ket{\emptyset}$
by $\{h^\perp_r\}_{r>0}$, formula \eqref{Lhr}
allows us to compute 
$L_{\bra{\emptyset},w_2}(u)$ for all $w_2\in W$
inductively. 
Formula \eqref{Len} with $n\le 0$ then 
allows us to compute $L_{w_1,w_2}(u)$ for all $w_1\in W^*$
and $w_2\in W$, because
$W^*=\F(v)^*$ is generated from $\bra{\emptyset}$ by
$\{e^\perp_n\}_{n\le 0}$
(note that $e_n$ ($n\ge2$) and $[e_0,e_2]$ 
are generated by $e^\perp_{-r}$ ($r>0$) and $\ad e^\perp_0=\ad h_1$).  
In general, $L_{w_1,w_2}(u)$ is expressed as 
$\sum a_i(u) x_i L_{\bra{\emptyset},\ket{\emptyset}}(u) y_i$
with some $x_i,y_i\in \B^\perp$ and $a_i(u)\in\C[u]$. 
Hence $L_{w_1,w_2}(u)$ is polynomial. 
Using 
\begin{align*}
&u^\partial h^\perp_r u^{-\partial}M[m]\subset 
u^r\cdot \sum_{j=m}^{m+r}M[j]u^{-m-r+j}
\quad (r>0)\,,
\\
&u^\partial e^\perp_n u^{-\partial}M[m]\subset \delta_{n,0}M[m]+M[m-1]u^{-1}
\quad (n\le0)\,,
\end{align*}
we can show \eqref{growthM} by induction. 
When $n=0$, we have to be careful  
with the first term in the right hand side of \eqref{Len}.
In this case we write it 
as $u[L_{w_1,w_2}(u),\bar{h}_{1,V}]$ and use \eqref{hbarV}.
\end{proof}

\medskip

\noindent{\it Remark.}\quad 
While matrix elements of 
$\Rb_{M^+(1),\F(v)}(u)$ are polynomials, its inverse
$\Rb_{M^+(1),\F(v)}(u)^{-1}$ has
infinitely many poles. In fact, its diagonal matrix entry 
on the vector $v_0\otimes \ket{\lambda}$ is given by
$\prod_{\square\in\lambda}\bigl(1-q^\square u/v\bigr)^{-1}$. \qed
\bigskip

Next we consider the $2$-finite module $N^+(1)$. 
Recall the grading 
$N^+(1)=\oplus_{m=0}^\infty N[m]$, 
$N[m]=N_0[m]\oplus N_1[m]$, given in Proposition \ref{prop:gradingN}. 
Let $u^\partial$ be the operator $u^\partial\bigl|_{N[m]}=u^m\cdot \id_{N[m]}$. 

\begin{prop}\label{prop:poly-N}
The operator
$(1-q_2u/v)L_{w_1,w_2}(u)$ is polynomial for any $w_1\in \F(v)^*$
and $w_2\in \F(v)$. Moreover 
we have the expansions as $u\to\infty$ 
\begin{align}
&u^\partial L_{w_1,w_2}(u)u^{-\partial}\cdot N_1[m]
\subset u^{2s} \sum_{j=1}^{m+2s}N_1[j][[u^{-1}]]\,,
\label{LN1}\\
&u^\partial L_{w_1,w_2}(u)u^{-\partial}\cdot N_0[m]
\subset u^{2s} \sum_{j=0}^{m+2s}N_0[j][[u^{-1}]]
+ u^{2s-1} \sum_{j=1}^{m+2s-2}N_1[j][[u^{-1}]]\,
\label{LN0}
\end{align}
where $s=-\pdeg w_2$.
\end{prop}
\begin{proof}
The proof is similar to that of Proposition \ref{prop:poly-M}. 

Denote by $\br{h_r}_i$ the eigenvalue of $h_r$ on $N_i$, and set
$\bar{h}_{r,i}=h_r-\br{h_r}_i$. 
Then we have 
\begin{align*}
L_{\bra{\emptyset},\ket{\emptyset}}(u)
&=
q^{-d^\perp}\exp\Bigl(-\sum_{r>0}\Bigl(\frac{u}{v}\Bigr)^{r}
(1-q_2^r)\bar{h}_{r,0}\Bigr)\,,\\
&=\frac{1-u/v}{1-q_2u/v}
q^{-d^\perp}\exp\Bigl(-\sum_{r>0}\Bigl(\frac{u}{v}\Bigr)^{r}
(1-q_2^r)\bar{h}_{r,1}\Bigr)\,.
\end{align*}
It follows from \eqref{psiN} that
$(1-q_2u/v)L_{\bra{\emptyset},\ket{\emptyset}}(u)$ is polynomial,  
and that 
$u^r u^\partial \bar{h}_{r,i}u^{-\partial}$ is 
independent of $u$ on $N_i$. 
Hence \eqref{LN1}, \eqref{LN0} are true in this case. 

The case of general $w_1,w_2$ can be verified by using \eqref{Lhr}, 
\eqref{Len} and the relations which follow from 
Proposition \ref{prop:gradingN},
\begin{align*}
&u^\partial h^\perp_r u^{-\partial}N_0[m]
\subset u^{2r}\sum_{j=0}^{m+2r}N_0[j]\otimes\C[u^{-1}]
+u^{2r-1}\sum_{j=1}^{m+2r-1}N_1[j]\otimes\C[u^{-1}]\,,
\\
&u^\partial h^\perp_r u^{-\partial}N_1[m]
\subset u^{2r}\sum_{j=1}^{m+2r}N_1[j]\otimes\C[u^{-1}]\,,
\\
&u^\partial e^\perp_{-n} u^{-\partial}N_0[m]
\subset u^{-1}N_0[m-1]\,,
\quad (n>0)\\
&u^\partial e^\perp_{-n} u^{-\partial}N_1[m]
\subset u^{-1}N_1[m-1]
+\sum_{j=0}^m N_0[j]\otimes\C[u^{-1}]
\quad (n>0)\,,
\\
&u^\partial \bar{h}_{r,i} u^{-\partial}N_i[m]\subset
u^rN_i[m-r]\,.
\end{align*}
\end{proof}

\subsection{Bethe ansatz}

For an object $V\in \cO_{\B^\perp}$, 
the twisted transfer matrix 
associated with the `auxiliary space' $V$ 
is a formal series defined by
\begin{align*}
\Tb_V(u;p)&=\Tr_{V,1}\left((p^{-d^\perp}\otimes \id)\cdot \cR(u)\right)
\quad \in \Bb^\perp[[u,p]][p^{-1}]\,.
\end{align*}
Here $\Tr_{V,1}$ means that the trace is taken on the first tensor component.
Clearly we have 
\begin{align*}
&\Tb_{V_1\oplus V_2}(u;p)=\Tb_{V_1}(u;p)+\Tb_{V_2}(u;p)\,,
\\
&\Tb_{V_1\otimes V_2}(u;p)=\Tb_{V_2}(u;p)\Tb_{V_1}(u;p)\,,
\end{align*}
hence the assignment
$V\mapsto \Tb_V$ gives a homomorphism of rings from 
$\Rep\,\B^\perp$ to $\Bb^\perp[[u,p]][p^{-1}]$. 

Element $\Tb_V(u;p)$ gives rise to 
a formal series of operators which act 
on any given `quantum space' $W\in\cO_\E$.  
It is convenient to use the normalized $R$ matrix and define
\begin{align*}
T_{V,W}(u;p)
&=\Tr_{V,1}\bigl((p^{-d^\perp}\otimes\id)\Rb_{V,W}(u)\bigr)\,
\quad \in \End(W)[[u,p]]\,,
\end{align*}
so that $\Tb_V(u;p)\bigl|_{W}=f_{V,W}(u)T_{V,W}(u;p)$. 
Note that $T_{V,W}(u;p)$ acts on each subspace of $W$ of fixed principal degree. 

From now on, we choose
\begin{align*}
W=\F(v_1)\otimes_{\Delta^\perp}\cdots\otimes_{\Delta^\perp} \F(v_N).  
\end{align*}
Note that $C^\perp$ acts as $q^{-N}$ on $W$. 
We set 
\begin{align*}
&a(u)=\prod_{i=1}^Nq^{-1}\Bigl(1-\frac{q_2 u}{v_i}\Bigr)\,,
\quad
d(u)=\prod_{i=1}^N\Bigl(1-\frac{u}{v_i}\Bigr)\,,
\end{align*}
and introduce the notation 
\begin{align*}
&Q_W(u;p)=T_{M^+(1),W}(u;p)\,,
\quad
\cT_W(u;p)=a(u)T_{N^+(1),W}(u;p)\,.
\end{align*}
In what follows $W$ is fixed, we drop it from notation and simply write $Q(u;p)$ and $\cT(u;p)$.

\begin{prop}\label{Q pol}
On the subspace $W_{-k}$ of principal degree $-k$, $Q(u;p)$ is a polynomial in $u$
of degree at most $k$. 
\end{prop}
\begin{proof}
Let $Q_n(u)=\Tr_{M^+(1)_{-n}}\Rb_{M^+(1),W}(u)$ denote 
the trace over the subspace  $M^+(1)_{-n}$ of fixed principal degree $-n$. 
We have $Q(u;p)=\sum_{n=0}^\infty p^n Q_n(u)$.  
Since
\begin{align*}
\Rb_{M^+(1),W}(u)= 
\Rb_{M^+(1),\F(v_N)}(u)\cdots \Rb_{M^+(1),\F(v_1)}(u), 
\end{align*} 
each matrix element of $\Rb_{M^+(1),W}(u)$ is a polynomial in $u$ 
by Proposition \ref{prop:poly-M}. 
Moreover \eqref{growthM} implies that  
for each vector $w\in W_{-k}$ the degree of   
$Q_n(u)w$ does not exceed $k$. Therefore 
$Q(u;p)w$ is also a polynomial in $u$ of degree at most $k$. 
\end{proof}

\begin{prop}\label{T pol}
On  each subspace of $W$ of fixed principal degree, 
$\cT(u;p)$ is a polynomial in $u$.
\end{prop}
\begin{proof}
This follows from Proposition  \ref{prop:poly-N}. 
\end{proof}
\medskip

Since the map
$V\mapsto \Tb_V$ is a ring homomorphism, 
relations in $\Rep\,\B^\perp$ implies those for the transfer matrices. 
It allows us to express any transfer matrix 
$T_{V,W}(u;p)$ via the operator $Q(u;p)$.
The recipe is as follows. 
\begin{prop}\label{prop:TQ}
Let $\Psi\in\rg_{\B^\perp}$ and $V=L(\Psi)$. 
Then its $q$-character $\chi_q(V)$ has the form \eqref{1+A}.  
In $\chi_q(V)$, drop $\bar \chi_0^{\max(d(\Psi,0)}$ and replace each $X_a$ by $f_{M^+(1),W}(a)Q(a;p)$, 
each $x_a$ in $\bs m(\Psi)$ by $a^{-d^{\perp}}$, and each $t^{-1}$ by $pq^N$. 
Then the resulting series coincides with $f_{V,W}(u)T_{V,W}(u;p)$. 
\end{prop}
\begin{proof}
In general, if $V=L(\Psi)$ with $\Psi(z)=a \prod_i (1-a_i/z)/\prod_j (1-b_j/z)$, 
then by Lemma \ref{lem:q-char V} its $q$-character takes the form 
\begin{align*}
\chi_q(V)=x_a \frac{\prod_i X_{a_i}}{\prod_j X_{b_j}}\Bigl(\sum_{\mathcal{S}} \prod_{a\in\mathcal{S}} (tA_{a})^{-1}\Bigr)
\ \chi_0^{\max(d(\Psi),0)}\,.
\end{align*}
Since $\chi_q$ is injective, we have an identity in $\Rep\,\B^\perp$ 
\begin{align}
[V]\equiv [L(a)]\frac{\prod_i [M^+(a_i)]}{\prod_j [M^+(b_j)]}
\sum_{\mathcal{S}} \Bigl(\prod_{a\in\mathcal{S}} \prod_{s=1}^3 \frac{[M^+(q_s a)]}{[M^+(q_s^{-1} a)]}\Bigr)\{-\sharp \mathcal{S}\}
\,,
\label{[V]}
\end{align}
valid modulo $F^m\subset \Rep_0\, \B^\perp$ for any given $m<0$. This identity means the one obtained by clearing all denominators. 
Upon taking trace over both sides, we note that 
the degree shift $\{-d\}$ produces a power $(pq^N)^d$. Note also that $C^\perp_V=a^{-1}$ brings about $a^{-d^\perp}$ through $q^{t_\infty}$. 
We obtain the assertion.
\end{proof}

\begin{lem}\label{lem:a-fac}
By the above rule, the term $t^{-1}A_{a}^{-1}$ is replaced by
\begin{align*}
\mathfrak{a}(u;p)=p\,
\frac{d(u)}{a(u)}\prod_{s=1}^3\frac{Q(q_su;p)}{Q(q_s^{-1}u;p)} \,.
\end{align*}
\end{lem}
\begin{proof}
We compute the scalar factor.  
Substituting $\kappa_r\br{h_r}_{M^+(1)}=-1/r$ into \eqref{fVW} we find
\begin{align*}
pq^N\prod_{s=1}^3\frac{f_{M^+(1),W}(q_s u)}{f_{M^+(1),W}(q_s^{-1} u)}
&=pq^N\exp\Bigl(-\sum_{r>0}\kappa_r\br{h_{-r}}_Wu^r\Bigr)\\
&=pq^N\br{(\psi_0^-)^{-1}\psi^-(u)}_W^{-1}\\
&=p\frac{d(u)}{a(u)}\,.
\end{align*}
\end{proof}

For example, if $V=\F(1)$ is a Fock space, then from \eqref{F q-char} and Lemma \ref{lem:a-fac}
we obtain
\begin{align*}
T_{\F(1),W}(u;p)=&\frac{Q(q_2^{-1}u;p)}{Q(u;p)}
\sum_{\lambda\in\cP}
\prod_{\square\in\lambda}\mathfrak{a}(q^{-\square} u;p)\,.
\end{align*}

We remark that a formula for the operator $T_{\F(1),W}(u;p)$ in the language of Shuffle algebras
was obtained in \cite{FT2} . 

In particular, Proposition \ref{prop:TQ} 
allows to compute the spectrum of $T(V)$ 
from that of $Q(u;p)$. 

\begin{cor}\label{V Q w}
Let  $w$ be an eigenvector of $Q(u;p)$. 
Then the eigenvalue of $T_{V,W}(u;p)$ is given by the recipe of Proposition \ref{prop:TQ},  
replacing $Q(u;p)$ by its eigenvalue $Q_w(u;p)$ on $w$.
\qed
\end{cor}
\medskip

Finally, we show that zeroes of eigenvalues of $Q_w(u;p)$ satisfy the Bethe ansatz equation.

Let $w\in W_{-k}$ be an eigenvector of  $Q(u;p)$ of principal degree $-k$. 
By Proposition \ref{Q pol},  $Q_w(u;p)$ is a polynomial of degree at most $k$, so it has the form 
\begin{align*}
\frac{Q_w(u;p)}{Q_w(0;p)}=\prod_{i=1}^k(1-u/\zeta_i(p;w))\,,
\end{align*}
where some of $\zeta_i(p;w)$'s may be $\infty$. 
We conjecture that all roots are in fact finite. 

\begin{thm}\label{thm:eigv}
The zeroes  $\zeta_i(p;w)$, $i=1,\cdots,k$, of the eigenvalue of $Q_w(u;p)$ 
satisfy the Bethe ansatz equations
\begin{align*}
a(\zeta_i(p;w)) \prod_{s=1}^3 Q_w(q_s^{-1}\zeta_i(p;w);p)
+
p d(\zeta_i(p;w)) \prod_{s=1}^3 Q_w(q_s \zeta_i(p;w);p)=0\,.
\end{align*} 
\end{thm}
\begin{proof}
By Proposition \ref{T pol}, $\cT(u;p)$ is a polynomial.
Equation \eqref{NM-MMM} in $\Rep\,\B^\perp$ implies that 
\begin{align*}
\cT(u;p)Q(u;p)&=a(u) \prod_{s=1}^3 Q(q_s^{-1}u;p)
+
p d(u) \prod_{s=1}^3 Q(q_s u;p)\,.
\end{align*}
We apply this relation to $w$ and substitute $u=\zeta_i(p;w)$.
The theorem follows.
\end{proof}

\appendix
\section{}
We collect here technical lemmas used in the main text. 
Recall that we drop $D$ and set $C=1$ throughout. 

\subsection{Coproduct $\Delta^\perp$ on horizontal generators}

The following is an analog of a standard calculation in quantum affine 
algebras. 

\begin{lem}\label{Delta-psi}
The following equalities hold.
\begin{align}
\Delta^\perp e_n
&\equiv 
\sum_{0\le p<n}e_{n-p}\otimes \psi^+_p +1 \otimes e_n 
\quad 
\bmod \mathcal{E}_{\gge 2}\otimes\mathcal{E}_{\lle-1}
\quad (n>0),
\label{Dperp-e1}\\
\equiv 
&\sum_{n\le p\le 0} e_{n-p}
\otimes \psi^-_{p} +1\otimes e_{n}
\quad 
\bmod \mathcal{E}_{\gge 2}\otimes\mathcal{E}_{\lle-1}
\quad (n\le 0),
\label{Dperp-e2}\\
\Delta^\perp f_n
\equiv 
&\sum_{0\le p\le n} \psi^+_p\otimes f_{n-p}
+f_n\otimes 1
\quad 
\bmod \mathcal{E}_{\gge 1}\otimes\mathcal{E}_{\lle-2}
\quad  (n\ge 0),
\label{Dperp-f1}\\
\equiv 
&\sum_{n< p\le 0} \psi^-_{p} \otimes f_{n-p}
+f_n\otimes 1
\quad 
\bmod \mathcal{E}_{\gge 1}\otimes\mathcal{E}_{\lle-2}
\quad (n< 0),
\label{Dperp-f2}\\
\Delta^\perp\psi_n^+\equiv 
&\sum_{0\le p \le n} \psi^+_{p}\otimes \psi^+_{n-p}
\quad 
\bmod \mathcal{E}_{\gge 1}\otimes\mathcal{E}_{\lle-1}
\quad (n>0),
\label{Delta-hp}\\
\Delta^\perp\psi_{n}^-\equiv 
&\sum_{n\le p\le 0} 
\psi^-_{n-p}\otimes \psi^-_{p}
\quad 
\bmod \mathcal{E}_{\gge 1}\otimes\mathcal{E}_{\lle-1}
\quad (n<0)\,.
\label{Delta-hm}
\end{align}
\end{lem}
\begin{proof}
Since the calculation is well known for quantum affine algebras, 
we will be brief. 

We use $e_1=e^\perp_{-1}$, 
$e_0=h^\perp_{-1}$, 
$[e_n,h_1]=e_{n+1}$, 
$[h_{-1},e_n]=e_{n-1}$, 
and the relations
\begin{align*}
&\Delta^\perp e^\perp_n=
\sum_{j\ge0}e^\perp_{n-j}\otimes \psi_j^{+,\perp}(C^\perp)^n+
1\otimes e_n^\perp
\quad (n\in \Z)\,,
\\
&\Delta^\perp h_1\equiv h_1\otimes 1+e_1\otimes \kappa_1 f_0
+1\otimes h_1\quad \bmod \E_{\gge2}\otimes\E_{\lle-2}\,,
\\
&\Delta^\perp h_{-1}\equiv h_{-1}\otimes 1+
\kappa_1e_0\otimes f_{-1} 
+1\otimes h_{-1}
\quad \bmod \E_{\gge2}\otimes\E_{\lle-2}\,.
\end{align*}
Noting further that 
$\kappa_1[e_n,f_0]=\psi^+_n$ ($n>0$) and
$\kappa_1[f_{-1},e_{-n+1}]=\psi^-_{-n}$ ($n>0$), 
we obtain \eqref{Dperp-e1} and \eqref{Dperp-e2} by induction. 

The automorphism $\theta^2$ in \eqref{th2} is an anti-automorphism 
of coalgebras with respect to $\Delta^\perp$.  
Applying $\theta^2\otimes \theta^2$ to both sides
of  \eqref{Dperp-e1}, \eqref{Dperp-e2}, 
we obtain \eqref{Dperp-f1}, \eqref{Dperp-f2}. 

From \eqref{Dperp-e1} and \eqref{Dperp-f1} we compute
\begin{align*}
[\Delta^\perp e_n, \Delta^\perp f_0] 
\equiv \sum_{0\le p<n}[e_{n-p},f_0]\otimes\psi^+_p
+\psi^+_0\otimes[e_n,f_0]\quad
\bmod \E_{\gge1}\otimes\E_{\lle-1}\,,
\end{align*}
which leads to \eqref{Delta-hp}. 
Finally we get  \eqref{Delta-hm}
by applying $\theta^2\otimes\theta^2$
to \eqref{Delta-hp}. 
\end{proof}

\subsection{Polynomiality of currents}\label{subsec:poly-cur}

In this subsection we prove Proposition \ref{prop:poly-cur}.

Recall the subalgebras $\B^\perp_\pm$ given in \eqref{Bpm}. 
In the following computation we use also auxiliary subalgebras
\begin{align*}
&\cN^\perp=\langle e^\perp_n\ (n\in \Z),\ C^\perp \rangle\,,
\\
&\cN_{-}^\perp=\langle \bp_{(\nu_1,\nu_2)}\ (\nu_1<0,\nu_2>0)
 \rangle\,,
\quad \cN_{\lle}^\perp=\cN_-^\perp\B_0^\perp\,.
\end{align*}
Algebra $\B^\perp_-$ is generated by $f_n$ ($n\ge0$) and 
$\psi^{+,\perp}_r$ ($r>0$). 
Algebra $\B^\perp_+$ is generated by elements 
$e_n$ ($n>0$), $e^\perp_{-r}$ ($r>0$)  and $[e_0,e_2]$.
Likewise $\cN_{-}^\perp$ is generated by 
$f_n$ ($n>0$), $e^\perp_{r}$ ($r>0$)  and $[f_0,f_2]$.

We make extensive use of formulas expressing
$\psi^{+,\perp}_r$, $e^\perp_n$ in terms of the horizontal ones,  
\begin{align}
&\psi^{+,\perp}_r=\kappa_1
\bigl(-C^\perp\bigr)^{r-1}\cdot \ad f_{-1}(\ad f_0)^{r-2}f_1
\quad (r\ge2)\,,
\label{psi-ff}\\
&e^\perp_r=(-1)^{r-1}\bigl(C^\perp\bigr)^r\cdot (\ad f_0)^{r-1}f_1
\quad (r\ge1)\,,
\label{e-ff}\\
&e_{-r}^\perp=(\ad e_0)^{r-1}e_1\quad (r\ge1)\,.
\label{e-ee}
\end{align}
In addition we have $\psi^{+,\perp}_1=\kappa_1f_0$. 
We should note that 
the generators $e_0,f_{-1}$ do not belong to $\B^\perp$,   
only the commutators with them do.  
Namely we have 
\begin{align}
\ad e_0(\cN^\perp)\subset\cN^\perp\,,
\quad
\ad f_0(\cN^\perp_{\lle})\subset \cN^\perp_{\lle}\,,
\quad
\ad f_{-1}(\cN^\perp_{\lle})\subset \B^\perp\,.
\label{cN}
\end{align}

We begin with a technical lemma. 
\begin{lem}\label{lem:psi-X}
Let $r$ be a positive integer. 

(i) For any $\ell \ge1$, there exists an $N\in \Z_{>0}$ such that for all $m\ge N$ 
and $x_m=e_m,f_m,\psi^+_m$ we have
\begin{align}
[x_m,e^\perp_r] \in 
\sum_{j\ge \ell}\cN^\perp_{\lle}\cdot f_j+\sum_{j\ge \ell}\cN^\perp_{\lle}\cdot\psi^+_j\,.
\label{psi-X}
\end{align}

(ii) For any $a\in\C^\times$ and $\ell \ge1$ we have 
\begin{align}
[\psi^+(a),e^\perp_r]\in \sum_{j\ge \ell}\cN^\perp_{\lle}\cdot \psi^+(a)\cdot \cN^\perp_{\lle}\cdot f_j
+\cN^\perp_{\lle} \cdot \psi^+(a)\cdot 
\bigl(\cN^\perp_{\lle}\cap\E_{\gge -r+1}\bigr)\,. 
\label{psia}
\end{align}
\end{lem}
\begin{proof}
In algebra $\E$, we have the quadratic relations 
\begin{align}
&[\psi^+_m,f_n]=\bar{\xi}(\psi^+_{m-1}f_{n+1}+f_{n+2}\psi^+_{m-2})
-\xi(\psi^+_{m-2}f_{n+2}+f_{n+1}\psi^+_{m-1})
+[\psi^+_{m-3},f_{n+3}]\,,
\label{fpsirel}\\
&[f_m,f_n]=\bar{\xi}(f_{m-1}f_{n+1}+f_{n+2}f_{m-2})
-\xi(f_{m-2}f_{n+2}+f_{n+1}f_{m-1})
+[f_{m-3},f_{n+3}]\,,
\label{ffrel}
\end{align}
which hold for all $m,n\in\Z$. Here 
we set $\xi=\sum_{i=1}^3q_i$, $\bar{\xi}=\sum_{i=1}^3q_i^{-1}$. 
We can apply the same equations repeatedly to terms of the form 
$\psi^+_{m-j}f_{n+j}$, $f_{m-j}f_{n+j}$ 
in the right hand side. 
For any given $\ell\ge n+1$, we obtain as a result 
\begin{align}
&[\psi^+_m,f_n]\in 
\sum_{j=\ell}^{\ell+2}\C \psi^+_{m+n-j}f_j
+\sum_{j=n+1}^{\ell+2}\C f_j\psi^+_{m+n-j}\,
\label{psi-fm}
\end{align}
and a similar equation wherein all $\psi^+_k$'s are replaced by $f_k$'s. 
From these, along with the relation $[e_m,f_n]=\kappa_1^{-1}\psi^+_{m+n}$,
it is clear that \eqref{psi-X} holds for $r=1$, $e_1^\perp=C^\perp f_1$,  
provided $m$ is sufficiently large. 

By induction, suppose statement (i) is true up to $r$, and  
consider $e^\perp_{r+1}=-C^\perp[f_0,e^\perp_r]$. We write
\begin{align}
[x_m,[f_0,e^\perp_r]]=[[x_m,f_0],e^\perp_r]
+[f_0,[x_m,e^\perp_r]]. 
\label{xf0e}
\end{align}
We apply relations of the type \eqref{psi-fm} to the first term in the right hand side of \eqref{xf0e}, 
and use the induction hypothesis. 
Given an $\ell$, we can choose an $\ell'$ so that
\begin{align*}
[[x_m,f_0],e^{\perp}_r]&\in \sum_{j\ge\ell'}\cN_{\lle}^\perp\cdot f_{j}
+\sum_{j\ge\ell'}\cN_{\lle}^\perp \cdot\psi^+_{j}+
\sum_{j\ge\ell'}\cN_{\lle}^\perp \cdot[f_{j},e^\perp_r]
+\sum_{j\ge\ell'}\cN_{\lle}^\perp \cdot[\psi^+_{j},e^\perp_r]
\\
&\subset \sum_{j\ge \ell}\cN_{\lle}^\perp\cdot f_j
+\sum_{j\ge \ell}\cN_{\lle}^\perp \cdot\psi^+_j\,
\end{align*}
holds for $m$ large enough. 
We proceed similarly with the second term of \eqref{xf0e},  
using the induction hypothesis, relation of type \eqref{psi-fm}, 
and \eqref{cN}. With a suitable $\ell''$ we obtain 
\begin{align*}
[f_0,[x_m,e^\perp_r]]&\in 
\sum_{j\ge\ell''}\cN_{\lle}^\perp\cdot f_j
+\sum_{j\ge\ell''}\cN_{\lle}^\perp \cdot\psi^+_j
+\sum_{j\ge\ell''}\cN_{\lle}^\perp \cdot[f_0,f_j]+
\sum_{j\ge\ell''}\cN_{\lle}^\perp \cdot[f_0,\psi^+_{j}]
\\
&\subset \sum_{j\ge \ell}\cN_{\lle}^\perp \cdot f_j
+\sum_{j\ge \ell}\cN_{\lle}^\perp \cdot\psi^+_j\,
\end{align*}
for $m$ large enough. This proves (i). 

To show (ii) for $r=1$ it is enough to use 
\begin{align}
[\psi^+(a),f_n]\in \sum_{j\ge\ell}\C\psi^+(a)f_j
+\sum_{j\ge n+1}\C f_j\psi^+(a) \,.
\label{psia-rel}
\end{align}
The induction step is also very similar to the argument given
above. The only thing to note is
that $[f_0,\cN^\perp_{\lle}\cap\E_{\gge -r+1}]\subset \cN^\perp_{\lle}\cap\E_{\gge -r}$. 
\end{proof}
The following is the content of Proposition \ref{prop:poly-cur}.
\begin{lem}\label{lem:poly-cur}
Let $M=L(\Psi)$, $\Psi\in\rg_{\B^\perp}\cap\C[z^{-1}]$. 
For any $v\in M$ we have 
\begin{align}
&\text{$e_nv=f_nv=\psi^+_nv=0$ for sufficiently large $n$,} 
\label{poly-st}\\
&\text{if $a\in\C^\times$ is a zero of $\Psi(z)$, then $\psi^+(a)v=0$.}
\label{psi-zero}
\end{align}
\end{lem}
\begin{proof}
We prove the assertion by induction on $-\pdeg v$.

When $v=v_0$ is the highest $\ell$-weight vector, \eqref{poly-st} and 
\eqref{psi-zero} are evident except for $f_nv$. 
Let us show $f_n v_0=0$ for $n>\deg_{z^{-1}} \Psi(z)$.  
To see this, note that 
$e_m f_nv_0=\kappa_1^{-1} \psi^+_{m+n}v_0=0$ ($m\ge1$), and that 
$e^\perp_{-m}f_nv_0=0$ ($m\ge2$), $[e_0,e_2]f_nv=0$
 for degree reasons. 
Hence $f_n v_0=0$ by the irreducibility of  $M$. 

Assume that \eqref{poly-st}, \eqref{psi-zero}
hold true for $v'\in M$ with $\pdeg v'>-l$, and take $v\in M$ with $\pdeg v=-l$.
There are two cases to consider, $v=f_rv'$ ($r\ge0$), 
or $v=\psi^{+,\perp}_rv'$ ($r\ge2$), 
for some $v'$. 
In the following let $x_n$ stand for one of $e_n,f_n,\psi^+_n$.  

If $v=f_rv'$, then $x_nv=[x_n,f_r]v'+f_rx_nv'$. 
For large $n$ this vanishes by \eqref{psi-fm}, its analog for $[f_m,f_n]$, 
and the induction hypothesis. 
We have also
\begin{align*}
&\psi^+(a)f_rv'\in f_r\psi^+(a)v'+\sum_{j\ge \ell}\C \psi^+(a)f_j v'+\sum_{j\ge r+1}\C f_j  \psi^+(a)v'
\end{align*}
for any $\ell$. 
The right hand side vanishes if $\ell$ is chosen large enough. 

Consider the case $v=\psi^{+,\perp}_r v'$. 
Since 
$(C^\perp)^{r-1}\psi^{+,\perp}_r=-\kappa_1[f_{-1},e^\perp_{r-1}]$, we have 
\begin{align*}
(C^\perp)^{r-1}[x_n,\psi^{+,\perp}_r]\in
\C[[x_n, f_{-1}],e^\perp_{r-1}]+\C[f_{-1},[x_n,e^\perp_{r-1}]]\,.
\end{align*}
Take any $\ell\ge1$. 
If $x_n=e_n$ then $[x_n,f_{-1}]\in\C\psi^+_{n-1}$. 
If  $x_n=f_n$ or $\psi^+_n$ then 
$[x_n,f_{-1}]\in \sum_{j> n/2}\C x_{n-j-1}f_j+\sum_{j\le n/2}\C f_jx_{n-j-1}$.
In either case, using 
Lemma \ref{lem:psi-X} (i) 
we see that if $n$ is large enough then 
\begin{align}
[[x_n, f_{-1}],e^\perp_{r-1}]v'
\in\sum_{j\ge \ell} \cN_{\lle}^\perp \cdot f_jv'
+\sum_{j\ge \ell} \cN_{\lle}^\perp \cdot \psi^+_jv'\,.
\label{xn-f-ep1}
\end{align}

Since $[f_{-1},\cN_{\lle}^\perp]\subset \B^\perp$, we have also
\begin{align}
[f_{-1},[x_n,e^\perp_{r-1}]]v'
&\in \sum_{j\ge\ell}\B^\perp f_jv'+\sum_{j\ge\ell}\B^\perp \psi^+_jv'
+\sum_{j\ge\ell}\cN_{\lle}^\perp [f_{-1},f_j]v'+\sum_{j\ge\ell}\cN_{\lle}^\perp [f_{-1}, \psi^+_j]v'
\,.
\label{xn-f-ep2}
\end{align}
We can choose $\ell$ large enough so that each term in the right hand side of 
\eqref{xn-f-ep1}, \eqref{xn-f-ep2} vanishes. 
Hence $x_n\psi^{+,\perp}_rv'=0$ for sufficiently large $n$.  

It remains to show \eqref{psi-zero}. 
By using quadratic relations and \eqref{psia-rel},
and arguing as above, 
we find 
\begin{align*}
&[\psi^+(a),\psi^{+,\perp}_r]v'
\in \sum_{j\ge \ell}\B^\perp \psi^+(a)\B^\perp\cdot f_jv'
+\B^\perp\psi^+(a) \bigl(\B^\perp\cap \E_{\gge -r+1}\bigr)v'\,
\end{align*}
for any $\ell$. 
Hence if $\ell$ is large enough, then 
the right hand side vanishes by the induction hypothesis. 
\end{proof}

\subsection{Coproduct $\Delta$ and $1$-finite modules}\label{subsec:Delta}

We restate and prove Proposition \ref{prop:Delta} as the following lemma. 

\begin{lem}\label{lem:Delta}

(i) Algebra $\B^\perp$ is a left coideal of $\E$ 
with respect to $\Delta$: $\Delta(\B^\perp)\subset \E\widehat{\otimes} \B^\perp$, where $\widehat{\otimes}$ means the 
completed tensor product with respect to the homogeneous degree.

(ii) Let $V$ be an $\E$ module, and let $M=L(\Psi)$ with $\Psi\in\rg_{\B^\perp}\cap\C[z^{-1}]$. 
Then $\Delta(\B^\perp)$ has a well-defined action on $V\otimes M$.  
\end{lem}

\begin{proof}
The images of horizontal generators by $\Delta$ are given by 
\begin{align*}
&\Delta e_n =\sum_{j\ge0}e_{n-j}\otimes \psi^+_j+1\otimes e_n\,\quad (n>0)
\\
&\Delta f_n=f_n\otimes 1+\sum_{j\ge0}\psi^-_{-j}\otimes f_{n+j}\,\quad (n\ge0).
\end{align*}
Thanks to Proposition \ref{prop:poly-cur},
these series terminate on each vector of $V\otimes M$.  
We study the images of the vertical generators $e^\perp_r,\psi^{+,\perp}_r$.
 
If $r>0$, then $\Delta e^\perp_{-r}=(\ad \Delta e_0)^{r-1}\Delta e_1$ by \eqref{e-ee}. It is  
the coefficient of $z_1^0\cdots z_{r-1}^0z_r^{-1}$ in
the multiple commutator
$(\ad A_1+\ad B_1)\cdots (\ad A_{r-1}+\ad B_{r-1})(A_r+B_r)$, 
where $A_j=e(z_j)\otimes \psi^+(z_j)$ and $B_j=1\otimes e(z_j)$. 
We write $\ad X=L(X)-R(X)$ where $L(X)$ (resp. $R(X)$) means multiplication by $X$ from the left
(resp. from the right). Expanding the product, and taking coefficients of terms of the type $B_j$, 
we obtain 
a linear combination of terms of the following form:
\begin{align}
&e(z_{j'_1})\cdots e(z_{j'_k})\otimes X_0Y^\pm_1X_1 Y^\pm_2\cdots Y^\pm_{k-1}X_{k-1}\times
\begin{cases}
Y^\pm_{k}X_{k}(e_1) & (j_k<r)\\
\psi^+(z_r)& (j_k=r)\,.
\end{cases}
\label{XYXY}
\end{align}
Here $0\le k\le r$, $1\le j_1<\cdots<j_k\le r$, 
$j_1',\cdots,j_k'$ is a permutation of $j_1,\cdots,j_k$, 
$X_s=(\ad e_0)^{j_{s+1}-j_{s}-1}$ ($j_0=0$, $j_{k+1}=r$), 
$Y_s^+=L(\psi^+(z_{j_s}))$, $Y_s^-=R(\psi^+(z_{j_s}))$.  
Since $(\ad e_0)(\cN^\perp)\subset \cN^\perp$, coefficients of \eqref{XYXY} in the remaining $z_i$'s
belong to $\E\widehat{\otimes} \B^\perp$. 

There is nothing to show about $e^\perp_0=h_1$ and $\psi^{+,\perp}_1=\kappa_1f_0$. 

We show next that $\Delta e^\perp_{r}$ with $r>0$ belongs to $\E\widehat{\otimes} \cN_{\lle}^\perp$.
If $r=1$ it is clear from $e^\perp_1=C^\perp f_1$. 
If $r\ge2$, then $\Delta e^\perp_{r}$ is proportional (with $C^\perp$ as coefficients) to 
\begin{align*}
[f_0\otimes 1+\psi^-_0\otimes f_0+\sum_{j\ge1}\psi^{-}_{-j}\otimes f_{j}, \Delta e^\perp_{r-1}].
\end{align*}
Since $\ad f_0(\cN^\perp_{\lle})\subset\cN^\perp_{\lle}$, we see by induction that all terms belong to 
$\E\otimes \cN_{\lle}^\perp$.

Finally $\Delta \psi^{+,\perp}_r$ ($r\ge2$) is proportional to
\begin{align*}
[f_{-1}\otimes 1+\psi^-_0\otimes f_{-1}+\sum_{j\ge0}\psi^{-}_{-j-1}\otimes f_{j}, \Delta e^\perp_{r-1}].
\end{align*}
Since  $\ad f_{-1}(\cN^\perp_{\lle})\subset\B^\perp$, all terms belong to $\E\otimes \B^\perp$.
\end{proof}

\subsection{Coproduct $\Delta$ and $\Delta^\perp$}
The following is an analog of Proposition 3.8 of \cite{EKP}, 
which relates the Drinfeld coproduct with 
the standard coproduct for quantum affine algebras. 

\begin{lem}\label{Delta-Delta}
In the completed tensor product $\E\widehat{\otimes}\E$, 
we have the identity
\begin{align*}
\Delta^{op}(x)=
(q^{t_\infty}\cR_{+}\cR_0)^{-1}
\cdot
\Delta^\perp(x) \cdot
q^{t_\infty}\cR_{+}\cR_0
=\cR_- 
\cdot
\Delta^{\perp,op}(x)\cdot
\cR_-^{-1}
\end{align*}
for any $x\in \E$.  
\end{lem}
\begin{proof}
The second equality is a consequence of the intertwining property
\eqref{int-R} of $\cR$. 
Since both sides are algebra homomorphisms, it suffices to check the identity 
on the generators $x=e_0,f_0,h_{\pm1}$. 

Let us consider the case $x=e_0=h^\perp_{-1}$. 
Recall that the factor $\cR_-$ has the form \eqref{cR-}. 
Since $\mathrm{pdeg}\,h^\perp_{-1}=1$,  we find 
\begin{align}
\cR_-\Delta^{\perp,op}(h^\perp_{-1}) \cR_-^{-1}
&
=\cR_-\Bigl(h^\perp_{-1}\otimes 1+C^\perp\otimes h^\perp_{-1}\Bigr)
\cR_-^{-1}
\label{RDel1}\\
&\in e_0\otimes 1+\E_{\lle0}\otimes \E\,.
\nn
\end{align}
Similarly, using \eqref{cR+} we compute
\begin{align}
\cR_0^{-1}\cR_{+}^{-1}q^{-t_\infty} \Delta^\perp(h^\perp_{-1})
q^{t_\infty}\cR_{+}\cR_0
&
\ \in\ \cR_0^{-1}\Bigl((C^\perp)^{-1}\otimes h^\perp_{-1}+\E_{\gge1}\otimes\E
\Bigr)\cR_0
\label{RDel2}\\
&=\sum_{j\ge0}\psi^{+}_j\otimes e_{-j}+\E_{\gge1}\otimes\E\,.
\nn
\end{align}
In the last line we use the identity
\begin{align*}
&\cR_0^{-1}\cdot (C^{\perp})^{-1}\otimes e(z)\cdot \cR_0
=\psi^+(z)\otimes e(z)\,,
\end{align*}
which follows from \eqref{R0}. 
Comparing these equations we find that \eqref{RDel1}, \eqref{RDel2}
are both equal to 
\begin{align*}
\sum_{j\ge0}\psi^{+}_j\otimes e_{-j}
+e_0\otimes 1=\Delta^{op}(e_0)\,.
\end{align*}
The other cases are proved in a similar manner, using 
\begin{align*}
&
\cR_0^{-1}\cdot f(z)\otimes C^\perp\cdot \cR_0
=f(z)\otimes \psi^-(z)\,.
\end{align*}
\end{proof}

\begin{cor}\label{cor:Delta-Delta}
Let $V$ be an object of $\cO_\E$, and let 
$M=L(\Psi)$ with $\Psi\in\rg_{\B^\perp}\cap\C[z^{-1}]$. Then 
we have an isomorphism of $\B^\perp$ modules
\begin{align*}
 V\otimes_{\Delta}M\simeq V\otimes_{\Delta^\perp}M\,.
\end{align*}
\end{cor}
\begin{proof}
We show that the element $\sigma(\cR_-)$, 
where $\sigma(a\otimes b)=b\otimes a$, has a well defined action on 
$V\otimes M$. To see this, let $M=\oplus_{m\ge0}M[m]$ 
be the grading of vector space mentioned in 
Remark at the end of Subsection \ref{subsec:positive}. 
Take vectors $v\in V$ and $w\in M[m]$. 
Then $\sigma(\cR)v\otimes w$ is a sum of terms of the form 
$y'_{\nu_1,-\nu_2}v\otimes x'_{-\nu_1,\nu_2}w$ (see \eqref{cR-}).
For them to be non-zero we must have 
$\nu_1+\pdeg v\le 0$ and $m-\nu_2+N\nu_1\ge0$, where
$N=\deg_{z^{-1}}\Psi$. Hence the sum is finite. 

Setting $F=\sigma(\cR_-)$
we have $F\Delta(x)=\Delta^\perp(x)F$ ($x\in\B^\perp$)
by Lemma \ref{Delta-Delta}. 
Therefore $F: V\otimes_{\Delta}M\to V\otimes_{\Delta^\perp}M$ 
gives the desired isomorphism.   
\end{proof}

\subsection{Submodules of $V \otimes_{\Delta}M$}

In this subsection we prove Proposition \ref{prop:submod}.  
In the following, we assume that 
\begin{align}
&V=L(\Psi_V)\,,\quad \Psi_V\in \rg_\E\,,
\label{assV}
\\
&M=L(\Psi_M)\,,\quad \Psi_M\in \rg_{\B^\perp}\cap\C[z^{-1}]. 
\label{assM}
\end{align}
Note that $M$ is $1$-finite 
(see Corollary \ref{q char 1}). 

\begin{lem}\label{lem:different}
Let $v\in V_{\Psi}$ be a non-zero vector of $\ell$-weight $\Psi$. 
Then $e_m v$, $f_mv$ are sums of $\ell$-weight vectors 
with $\ell$-weight different from $\Psi$. 
\end{lem}
\begin{proof}
 This follows from Lemma \ref{Y} and its analog for $e(z)$.
\end{proof}

\begin{lem}\label{lem:submod1} 
Let $W\subset V\otimes_{\Delta}M$ be a submodule. 
Then we have
\begin{align}
&(1\otimes x)W\subset W\quad (\forall x\in \B^\perp_+)\,,
\label{1otimesx}
\\
&\bigl(f_p\otimes 1\bigr)W\subset W\,,
\quad
\bigl(\sum_{j\ge0}\psi^-_{-j}\otimes f_{p+j}\bigr) W\subset W 
\quad (\forall p\ge0)\,.
\label{psi-otimes-f}
\end{align}
\end{lem}
\begin{proof}
Without loss of generality we may assume $W\neq 0$. 
To see \eqref{1otimesx} for $x=e_m$ ($m>0$), 
let $w\in W$ be a non-zero $\ell$-weight vector. 
We have
\begin{align*}
W\ni \Delta e_{>}(z)\cdot w=\bigl(e(z)\otimes\psi^+(z)\bigr)_{>}\cdot w+
\bigl(1\otimes e_{>}(z)\bigr) \cdot w\,.
\end{align*}
By Lemma \ref{lem:different}, the first term is a sum of terms
of $\ell$-weight different from that of $w$, 
while the second term has the same $\ell$-weight 
because $M$ is $1$-finite.  
Hence $\bigl(e(z)\otimes\psi^+(z)\bigr)_{>}\cdot w$, 
$\bigl(1\otimes e_{>}(z)\bigr) \cdot w$ both belong to $W$. 
Similarly, for $r\ge2$ we have
\begin{align*}
W\ni\, \Delta\Bigl(\bigl(\ad e_0\bigr)^{r-1}e_1\Bigr) w
=1\otimes \bigl(\ad e_0\bigr)^{r-1}e_1\cdot w+\cdots,
\end{align*}
where $\cdots$ is a sum of terms which involve at least one $e_i$ in
the first component. In view of Lemma \ref{lem:different}, 
we conclude that $(1\otimes e^\perp_{-r})w$ belongs to $W$. 
Furthermore it is clear that if $(1\otimes x)W\subset W$ then 
$(1\otimes [h_1,x])W\subset W$. This proves \eqref{1otimesx}. 

Proof of \eqref{psi-otimes-f} is similar
where we use
$\Delta f_{\ge}(z)
=f_{\ge}(z)\otimes 1+\bigl(\psi^-(z)\otimes f(z)\bigr)_{\ge}$. 
\end{proof}

\begin{lem}\label{lem:submod} 
We retain the assumptions \eqref{assV}, \eqref{assM}. 
Let $m_0$ be the highest $\ell$-weight vector of $M$. 
Then any submodule $W$ of $V\otimes_{\Delta}M$
must have the form $W=V^{(0)}\otimes M$ where  
\begin{align*}
V^{(0)}=\{v\in V\mid v\otimes m_0\in W\}.
\end{align*}
The highest $\ell$-weight vector $v_0\in V$ belongs to 
$V^{(0)}$ if and only if $V=V^{(0)}$. 
\end{lem}
\begin{proof}
First we show that 
\begin{align}
V^{(0)}\otimes M_{n}\subset W\quad (\forall n\le0)
\label{V1MW}
\end{align}
by induction on $n$. 
By definition \eqref{V1MW} is true for $n=0$. 
Assuming it for $n>-N$, we take $m\in M_{-N}$. 

Consider first the case $m=f_pm'$. We set
$p_0=\max\{p\mid f_pm'\neq0\}$.
There exists an $x\in \B^\perp$ such that $x f_{p_0}m'=m_0$. 
By \eqref{psi-otimes-f}, we have 
$\sum_{j=0}^{p_0-p}\psi^-_{-j}v\otimes f_{p+j}m'\in W$ for 
$v\in V^{(0)}$. 
Applying $1\otimes x$ to both sides and using \eqref{1otimesx},  
we find inductively for $p=p_0,p_0-1,\cdots$ 
that $\psi^-_{-j}v\in V^{(0)}$ ($0\le j\le p_0$)
and 
$v\otimes f_{p}m'\in W$ ($0\le p\le p_0$). 

Next we consider the case $m=\psi^\perp_rm'$, $\pdeg w'=-N+r$. 
Setting $w'=v\otimes m'$ where $v\in V^{(0)}$ and 
we have
\begin{align*}
W\ni\, \Delta[f_{-1},e^\perp_{r-1}]w'
=[f_{-1}\otimes1,\Delta e^\perp_{r-1}]w'
+[\psi^-_0\otimes f_{-1},\Delta e^\perp_{r-1}]w'
+\sum_{j\ge1}[\psi^-_{-j}\otimes f_{-1+j},\Delta e^\perp_{r-1}]w'\,.
\end{align*}
From the previous paragraph we see that the first and the third terms
belong to $W$. The second term has the form
$[\psi^-_0\otimes f_{-1},(\psi^-_0)^{r-1}\otimes e^\perp_{r-1}]w' 
+\cdots$, 
where $\cdots$ denote terms containing at least one $f_i$ in the first 
component. Arguing similarly as in Lemma \ref{lem:submod1} we obtain that
$(1\otimes [f_{-1},e^\perp_{r-1}])w'\in W$. This shows 
$V^{(0)}\otimes \psi^\perp_r m'\in W$. 

Let us show that $V^{(0)}\otimes M=W$.
Let $w\in W$ be an $\ell$-weight vector, 
and let $w=\sum_{r=1}^Nv_r\otimes m_r$
($v_r\in V$, $m_r\in M$) be an expression where
$\{v_r\}$ and $\{m_r\}$ are linearly independent sets. 
We show that $v_r\in V^{(0)}$ for all $r$ by induction on $N$.  

Assuming $\pdeg m_1\le\pdeg m_r$ ($r\ge2$), we
choose an $x\in \B^\perp_+$ such that $x m_1=m_0$.
By \eqref{1otimesx}, we have $\sum_{r=1}^N v_r\otimes x m_r\in W$. 
For degree reasons we have $x m_r=a_r m_0$ for some $a_r\in \C$, 
hence $\sum_{r=1}^N a_r v_r\in V^{(0)}$ where we set $a_1=1$.   
If $N=1$, then we are done. Suppose $N\ge2$. 
By \eqref{V1MW}
we have $\sum_{r=1}^N a_r v_r\otimes m_1\in W$, so that 
$\sum_{r=2}^Nv_r\otimes (m_r-a_rm_1)\in W$. 
Since $\{v_r\}_{r=2}^N$  and  $\{m_r-a_rm_1\}_{r=2}^N$
are linearly independent,
the induction hypothesis applies and we obtain
$v_r\in V^{(0)}$ for $2\le r\le N$. This in turn implies $v_1\in V^{(0)}$. 

Finally, \eqref{psi-otimes-f} implies $f_{\ge}(z)V^{(0)}\subset V^{(0)}$. 
From the proof of Lemma \ref{lem:simple-to-simple}, 
we obtain that $f(z)V^{(0)}\subset V^{(0)}$.  
Therefore $v_0\in V^{(0)}$ if and only if $V^{(0)}=V$. 

The proof is now complete. 
\end{proof}

\bigskip

\noindent
 {\bf Acknowledgments.}\quad
MJ would like to thank Fedor Smirnov 
for kind invitation and hospitality during his visit to UPMC.
He wishes to thank also David Hernandez and Masato Taki 
for stimulating discussions, and Ludwig Faddeev 
for his interest in this work.

The contribution of BF is within the framework of a subsidy granted to 
the HSE by the Government of the Russian Federation for the 
implementation of the Global Competitiveness Program.


Research of EM is partially supported by the Simons foundation grant.

EM and BF would like to thank Kyoto University
for hospitality during their visits when this work was started.

\end{document}